\documentclass[aop,MSNbibl,dvips]{arximspdf}
\usepackage{subeqn}

\doi{10.1214/12-AOP804} 
\volume{41}
\issue{5}
\pubyear{2013}
\firstpage{3051}
\lastpage{3080}

\makeatletter

\newcommand{\rrvert}{\vert}
\newcommand{\llvert}{\vert}
\newtheorem{theorem}{Theorem}
\newtheorem{corollary}[theorem]{Corollary}
\newtheorem{lemma}[theorem]{Lemma}
\newproclaim{remark}[theorem]{Remark}
\makeatother

\begin{document}
\begin{frontmatter}

\title{A multivariate Gnedenko law of large numbers}
\runtitle{law of large numbers}

\begin{aug}
\author[A]{\fnms{Daniel} \snm{Fresen}\corref{}\ead[label=e1]{daniel.fresen@yale.edu}\ead[label=u1,url]{http://www.danielfresen.com}}
\runauthor{D. Fresen}
\affiliation{Yale University}
\dedicated{In memory of Nigel J. Kalton}
\address[A]{Department of Mathematics\\
Yale University\\
10 Hillhouse Avenue\\
New Haven, Connecticut 06511-6810\\
USA\\
\printead{e1}\\
\printead{u1}} 
\end{aug}

\received{\smonth{11} \syear{2010}}
\revised{\smonth{5} \syear{2012}}

%
\begin{abstract}
We show that the convex hull of a large i.i.d. sample from an absolutely
continuous log-concave distribution approximates a predetermined convex
body in the logarithmic Hausdorff distance and in the Banach--Mazur distance.
For log-concave distributions that decay super-exponentially, we also have
approximation in the Hausdorff distance. These results are multivariate
versions of the Gnedenko law of large numbers, which guarantees
concentration of the maximum and minimum in the one-dimensional case.

We provide quantitative bounds in terms of the number of points and the
dimension of the ambient space.
\end{abstract}

%


\end{frontmatter}

\section{Introduction}

The Gnedenko law of large numbers~\cite{Gn} states that if $F$ is the
cumulative distribution of a probability measure $\mu$ on $\mathbb{R}$ such
that for all $\varepsilon>0$%
%
\begin{equation}
\lim_{t\rightarrow\infty}\frac{F(t+\varepsilon
)-F(t)}{1-F(t+\varepsilon)}%
=\infty \label{Gnedenko
condition},
\end{equation}
then there are functions $\delta$, $T$ and $\mathcal{P}$ defined on $%
\mathbb{N}$ with%
%
\begin{eqnarray}
\lim_{n\rightarrow\infty}\delta_{n} &=&0, \label{error}
\\
\lim_{n\rightarrow\infty}\mathcal{P}_{n} &=&1 \label{prob}
\end{eqnarray}
such that for any $n\in\mathbb{N}$ and any i.i.d. sample $(\gamma
_{i})_{1}^{n}$ from $\mu$, with probability~$\mathcal{P}_{n}$, we have%
\[
\bigl\llvert \max\{\gamma_{i}\}_{1}^{n}-T_{n}
\bigr\rrvert <\delta_{n}.
\]
We define $0/0=\infty$ to allow for the trivial case when $\mu$ has
bounded support. The condition (\ref{Gnedenko condition}) implies
super-exponential decay of the tail probabilities $1-F(t)$, that is,
for all $%
c>0 $,%
\[
\lim_{t\rightarrow\infty}e^{ct}\bigl(1-F(t)\bigr)=0.
\]
The converse is almost true and can be achieved if we impose some sort of
regularity on $F$. One such regularity condition is log-concavity; see
Section~\ref{background}. Of course all of this can be re-worded in
multiplicative form.  Provided $1-F(t)$ is regular enough and decays
super-polynomially, that is, for any $m\in\mathbb{N}$,%
\[
\lim_{t\rightarrow\infty}t^{m}\bigl(1-F(t)\bigr)=0,
\]
then (\ref{error}) and (\ref{prob}) hold, and with probability $\mathcal
{P}%
_{n}$,%
\[
\biggl\llvert \frac{\max\{\gamma_{i}\}_{1}^{n}}{T_{n}}-1\biggr\rrvert \leq \delta_{n}.
\]
Note that rapid decay of the left-hand tail provides concentration of
$\min
\{\gamma_{i}\}_{1}^{n}$, and that $[\min\{\gamma_{i}\}_{1}^{n},\max
\{\gamma_{i}\}_{1}^{n}]=\operatorname{conv}\{\gamma_{i}\}_{1}^{n}$.

In this paper we extend the Gnedenko law of large numbers to the
multivariate setting. We consider a large collection of i.i.d. random
vectors $\{x_{i}\}_{1}^{n}$ in $\mathbb{R}^{d}$ that follow a log-concave
distribution $\mu$ with density function $f$. The object of interest
is the
convex hull $P_{n}=\operatorname{conv}\{x_{i}\}_{1}^{n}$, which is called a random
polytope. It is shown that with high probability, $P_{n}$ approximates a
deterministic body $F_{1/n}$ called the floating body, which is what remains
of $\mathbb{R}^{d}$ after deleting all open half-spaces $\mathfrak{H}$ such
that $\mu(\mathfrak{H})<1/n$. As in the one-dimensional case, the way in
which $P_{n}$ approximates $F_{1/n}$ depends on how rapidly $\mu$ decays.
Of primary interest is a quantitative analysis in terms of the number of
points, and in this regard our results are essentially optimal; see
Section %
\ref{optimality section}.

The fact that the floating body can be used in order to model random
polytopes is well known in the setting where $\mu$ is the uniform
distribution on a convex body; see, for example,~\cite{BL} and \cite
{Ba}. Our
main contribution is to study this approximation in the more general setting
of log-concave measures. Unlike the former case, the objects that we study
can have many different shapes as $n\rightarrow\infty$ and are not limited
to lie within a bounded region of space.

The notion of a multivariate Gnedenko law of large numbers has also been
considered by Goodman~\cite{Go} in the setting of Gaussian measures on
separable Banach spaces. In his paper he shows that with probability $1$,
the Hausdorff distance between the sample $\{x_{i}\}_{1}^{n}$ and the
ellipsoid $\sqrt{2\log n} \mathcal{E}$ converges to zero as
$%
n\rightarrow\infty$, where $\mathcal{E}$ is the unit ball of the
reproducing kernel Hilbert space associated to $\mu$.

\section{Main results}

Let $d\geq1$, $n\geq d+1$ and let $\mu$ be a log-concave probability
measure on $\mathbb{R}^{d}$ with a density function $f=d\mu/dx$. This means
that $f$ is of the form $f(x)=\exp(-g(x))$ where $g$ is convex. Let $%
(x_{i})_{1}^{n}$ denote a sequence of i.i.d. random vectors in $\mathbb
{R}%
^{d}$ with distribution $\mu$, and consider the random polytope $%
P_{n}=\operatorname{conv}\{x_{i}\}_{1}^{n}$. For any $x\in\mathbb{R}^{d}$, define%
\[
\widetilde{f}(x)=\inf_{\mathfrak{H}}\mu(\mathfrak{H}),
\]
where $\mathfrak{H}$ runs through the collection of all open half-spaces
that contain $x$. For any $\delta>0$, the \textit{floating body} is defined
as%
%
\begin{equation}
F_{\delta}=\bigl\{x\in\mathbb{R}^{d}\dvtx \widetilde{f}(x)\geq
\delta\bigr\} \label{floating body}.
\end{equation}
Note that $F_{\delta}$ is the intersection of all closed half-spaces $%
\mathfrak{H}$ such that $\mu(\mathfrak{H})\geq1-\delta$, and is
therefore convex. If $\mathfrak{H}$ is any open half-space that
contains the
centroid of~$\mu$, then $\mu(\mathfrak{H})\geq e^{-1}$ (see Lemma
5.12 in
\cite{LV} or Lemma 3.3 in~\cite{Bo}); hence $F_{\delta}$ is nonempty
provided that $\delta\leq e^{-1}$. Such a floating body was defined by
Sch%
\"{u}tt and Werner~\cite{SW} in the case where $\mu$ is the uniform
distribution on a convex body. We define the \textit{logarithmic Hausdorff
distance} between convex bodies $K,L\subset\mathbb{R}^{d}$ as
\[
d_{\mathfrak{L}}(K,L)=\inf\bigl\{\lambda\geq1\dvtx \exists x\in \operatorname{int}(K\cap L),
\lambda^{-1}(L-x)+x\subset K\subset\lambda(L-x)+x\bigr\},
\]
where we use the convention that $\inf(\varnothing)=\infty$. The main
result of the paper is as follows:

\begin{theorem}
\label{main 1}There exist universal constants $c,c^{\prime},\widetilde{c}>0$
with the following property. Let $q\geq1$, $d\in\mathbb{N}$ and $n\geq
c\exp\exp(5d)+c^{\prime}\exp(q^{3})$. Let $\mu$ be a probability
measure on $\mathbb{R}^{d}$ with a log-concave density function, $%
(x_{i})_{1}^{n}$ an i.i.d. sample from $\mu$, $P_{n}=\operatorname{conv}\{x_{i}\}_{1}^{n}$
and $F_{1/n}$ the floating body as in (\ref{floating body}). With
probability at least $1-3^{d+3}(\log n)^{-q}$,%
%
\begin{equation}
d_{\mathfrak{L}}(P_{n},F_{1/n})\leq1+\widetilde{c}d(d+q)
\frac{\log
\log n}{%
\log n} \label{main bound}.
\end{equation}
\end{theorem}

The strategy of the proof is to use quantitative bounds in the
one-dimen\-sional case to analyze the dual Minkowski functional of
$P_{n}$ in different
directions. The idea is simple; however, there are some subtle
complications. The lack of symmetry is a complicating factor, and the fact
that the half-spaces of mass $1/n$ do not necessarily touch $F_{1/n}$ adds
to the intricacy of the proof.

We define $f$ to be $p$-log-concave if it is of the form $f(x)=c\exp
(-g(x)^{p})$ where $g$ is a nonnegative convex function and $c>0$.

\begin{theorem}
\label{main 2}For all $q>0$, $p>1$ and $d\in\mathbb{N}$, and any
probability measure $\mu$ on $\mathbb{R}^{d}$ with a $p$-log-concave
density function, there exist $c,\widetilde{c}>0$ such that for all
$n\in
\mathbb{N}$ with $n\geq d+2$, if $(x_{i})_{1}^{n}$ is an i.i.d.\vadjust{\goodbreak} sample
from $%
\mu$, $P_{n}=\operatorname{conv}\{x_{i}\}_{1}^{n}$ and $F_{1/n}$ is the floating body as
in (\ref{floating body}), then with probability at least $1-\widetilde
{c}%
(\log n)^{-q}$ we have%
%
\begin{equation}
d_{\mathcal{H}}(P_{n},F_{1/n})\leq c\frac{\log\log n}{ ( \log
n )
^{1-{1}/{p}}}
\label{main hausdorff bound}.
\end{equation}
\end{theorem}

Theorem~\ref{main 2} can easily be extended to a much larger class of
log-concave distributions. Using Theorem~\ref{main 1}, any bound on the
growth rate of $\operatorname{diam}(F_{1/n})$ automatically transfers to a bound on
$d_{%
\mathcal{H}}(P_{n},F_{1/n})$.

Our prototypical example is the class of distributions introduced by
Schechtman and Zinn~\cite{SZ} of the form $f(x)=c_{p}^{d}\exp
(-\Vert x\Vert _{p}^{p})$, where $1\leq p<\infty$ and $c_{p}=p/(2\Gamma
(p^{-1}))$. For these distributions, $P_{n}\approx(\log
(c_{p}^{d}n))^{1/p}B_{p}^{d}$. Of particular interest is the Gaussian distribution, where $p=2$. In this
case (actually for the standard Gaussian distribution), B\'{a}r\'{a}ny and
Vu~\cite{BV} obtained a similar approximation (see Remark 9.6 in their
paper) and showed that there exist two radii, $R$ and $r$, both
functions of $%
n$ and $d$, such that for any fixed $d\geq2$ both $r,R=(2\log
n)^{1/2}(1+o(1))$ as $n\rightarrow\infty$, and with ``high
probability'' $%
rB_{2}^{d}\subset P_{n}\subset RB_{2}^{d}$. Their sandwiching result served
as a key step in their proof of the central limit theorem for Gaussian
polytopes (asymptotic normality of various functionals such as the volume
and the number of faces).

In the setting where $\mu$ is the uniform distribution on a convex body,
the floating body is usually denoted by $K_{\delta}$. In this context
it is
trivial that $\lim_{n\rightarrow\infty}d_{\mathcal{H}}(P_{n}, K)=0$ (almost
surely), and the phenomenon of interest is the rate at which $P_{n}$
approached the boundary of $K$. B\'{a}r\'{a}ny and Larman~\cite{BL} proved
that for $n\geq n_{0}(d)$,%
\[
c^{\prime}\operatorname{vol}_{d}(K\setminus K_{1/n})\leq
\mathbb{E}\mathrm {vol}%
_{d}(K\setminus P_{n})\leq
c^{\prime\prime}(d)\operatorname{vol}%
_{d}(K\setminus
K_{1/n}).
\]
The reader may be interested to contrast our results with the results in
\cite{DGT}. The results presented here require a very large sample size and
guarantee a precise approximation, somewhat in the spirit of the
``almost-isometric'' theory of convex bodies. On the other hand, the results
presented in~\cite{DGT} describe a type of approximation in the spirit of
the ``isomorphic'' theory, and are most interesting, specifically in high-dimensional spaces.

We also study two other deterministic bodies that serve as approximants to
the random body. Define%
\[
f^{\sharp}(x)=\inf_{\mathcal{H}}\int_{\mathcal{H}}f(y)\,d_{\mathcal{H}}(y),
\]
where $\mathcal{H}$ runs through the collection of all hyperplanes that
contain $x$, and $d_{\mathcal{H}}$ stands for Lebesgue measure on
$\mathcal{H%
}$. For any $\delta>0$, define the bodies%
\begin{eqnarray*}
D_{\delta}&=&\operatorname{Cl}\bigl\{x\in\mathbb{R}^{d}\dvtx f(x)\geq\delta\bigr
\},
\\
R_{\delta}&=&\operatorname{Cl}\bigl\{x\in\mathbb{R}^{d}\dvtx f^{\sharp}(x)
\geq\delta\bigr\},
\end{eqnarray*}
where $\operatorname{Cl}(E)$ denotes the closure of a set $E$. By log-concavity of $f$,
both $D_{\delta}$ and $R_{\delta}$ are convex.\vadjust{\goodbreak}

\begin{theorem}
\label{main 3}Let $d\in\mathbb{N}$, and let $\mu$ be a probability measure
on $\mathbb{R}^{d}$ with a continuous nonvanishing log-concave density
function. Then we have%
%
\begin{eqnarray}
\lim_{\delta\rightarrow0}d_{\mathfrak{L}}(F_{\delta},D_{\delta})
&=&1, \label{FD1}
\\
\lim_{\delta\rightarrow0}d_{\mathfrak{L}}(F_{\delta},R_{\delta})
&=&1. \label{FR1}
\end{eqnarray}
\end{theorem}

Similar results hold in the Hausdorff distance for log-concave distributions
that decay super-exponentially.

Let $X\in\mathbb{R}^{d}$ be a random vector with distribution $\mu$. The
random variable $-\log f(X)$ is a type of differential information
content; see~\cite{BM}. The differential entropy of $\mu$ is defined as
\begin{eqnarray*}
h(\mu) &=&-\mathbb{E}\log f(X)
\\
&=&-\int_{\mathbb{R}^{d}}f(x)\log f(x)\,dx
\end{eqnarray*}
and the entropy power defined as $N(\mu)=\exp(2d^{-1}h(\mu))$. Note that
the distribution of $-\log f(X)$ can be expressed in terms of the
function $%
\delta\mapsto\mu(D_{\delta})$,%
\[
 \mathbb{P}\bigl\{-\log f(X)\leq t\bigr\}=
\mu(D_{\delta})  \dvtx  \delta=e^{-t}.%
\]
Because of the rapid decay of $f$, the body $D_{\delta}$ acts as an
essential support for the measure $\mu$. For $\delta=e^{-d}$, this was
studied by Klartag and Milman~\cite{KM}; see Lemma~2.2 and Corollary
2.4 in
their paper. Bobkov and Madiman later provided a more precise description.
In~\cite{BM} they show that the variance of $-\log f(X)$ is at most $Cd$,
where $C>0$ is a universal constant, and that in high-dimensional
spaces, $%
f(X)^{2/d}$ is strongly concentrated around $N(\mu)$. Theorem 1.1 in their
paper can be written as%
\[
\mu\bigl\{x\in\mathbb{R}^{d}\dvtx N(\mu)^{-d/2}\delta<f(x)<N(
\mu)^{-d/2}\delta ^{-1}\bigr\}>1-2\delta^{p(d)}
\]
provided $\delta\in(0,1)$, where $p(d)=16^{-1}d^{-1/2}$. In Lemma \ref
{mass outside contour} we show that if $\mu$ is isotropic and has a
continuous density function, then for all $\delta<\exp(-10d\log d-7)$,
%
\begin{equation}
\mu\bigl\{x\in\mathbb{R}^{d}\dvtx f(x)\geq\delta\bigr\}\geq1-
\alpha_{d}\delta \bigl(\log \delta^{-1}\bigr)^{d},
\label{natural dependence on delta}
\end{equation}
where $\alpha_{d}=c_{1}\exp(3d^{2}\log d)$. In a fixed dimension,
inequality (\ref{natural dependence on delta}) displays the natural
quantitative behavior of $\mu(D_{\delta})$ as $\delta\rightarrow0$ and
is sharp up to a factor of $\log\delta^{-1}$.

Let $\mathcal{K}_{d}$ denote the collection of all convex bodies in
$\mathbb{%
R}^{d}$. For all $K,L\in\mathcal{K}_{d}$, define%
%
\begin{equation}
d_{\mathrm{BM}}(K,L)=\inf\bigl\{\lambda\geq1\dvtx \exists x\in
\mathbb{R}^{d},
\exists T, K\subset TL\subset
\lambda(K-x)+x\bigr\}, \label{affine BM}
\end{equation}
where $T$ represents any affine transformation of $\mathbb{R}^{d}$.
This is
a modification of the classical Banach--Mazur distance between normed spaces
(origin symmetric bodies).

\begin{theorem}
\label{main 4}For all $d\in\mathbb{N}$, there exists a probability
measure $%
\mu$ on $\mathbb{R}^{d}$ with the following universality property. Let
$%
(x_{i})_{1}^{\infty}$ be an i.i.d. sample from $\mu$, and for each
$n\in
\mathbb{N}$ with $n\geq d+1$, let $P_{n}=\operatorname{conv}\{x_{i}\}_{1}^{n}$. Then with
probability~1, the sequence $(P_{n})_{d+1}^{\infty}$ is dense in
$\mathcal{K%
}_{d}$ with respect to $d_{\mathrm{BM}}$.
\end{theorem}

Throughout the paper we will make use of variables $c$, $\widetilde{c}$,
$c_{1}$, $c_{2}$, $n_{0}$, $m$, etc. At times they represent
universal constants, and at other times they depend on parameters such
as the
dimension $d$ or the measure $\mu$. Such dependence will always be clear
from the context, and will either be indicated explicitly as $c_{d}$,
$c(d)$, $n_{0}(d)$, etc., or implicitly as in Theorem~\ref{main 2}, where $c$
and $%
\widetilde{c}$ depend on $q$, $p$, $d$ and~$\mu$. Half-spaces shall be
indexed as $\mathfrak{H}_{\theta,t}=\{x\in\mathbb{R}^{d}\dvtx  \langle
x,\theta \rangle\geq t\}$ and hyperplanes as $\mathcal{H}_{\theta,t}=\{x\in\mathbb{R}^{d}\dvtx  \langle x,\theta \rangle=t\}$,
where $%
\theta\in S^{d-1}$ and $t\in\mathbb{R}$.

\section{Background}\label{background}

Most of the material in this section is discussed in~\cite{Ball1,Ball2,Ma} and~\cite{MS}. We denote the standard Euclidean norm
on $%
\mathbb{R}^{d}$ by $\Vert \cdot\Vert _{2}$.  For any $\varepsilon>0$, an $%
\varepsilon$-net in $S^{d-1}$ is a subset $\mathcal{N}$ such that for any
distinct $\omega_{1},\omega_{2}\subset\mathcal{N}$, $\Vert \omega
_{1}-\omega
_{2}\Vert _{2}>\varepsilon$, and for all $\theta\in S^{d-1}$ there exists
$%
\omega\in\mathcal{N}$ such that $\Vert \theta-\omega\Vert _{2}\leq
\varepsilon$. Such a subset can easily be constructed using induction.  By a standard
volumetric argument, we have%
%
\begin{equation}
|\mathcal{N}|\leq \biggl( \frac{3}{\varepsilon} \biggr) ^{d}. \label{net
bound}
\end{equation}
By induction, any $\theta\in S^{d-1}$ can be expressed as a series%
%
\begin{equation}
\theta=\omega_{0}+\sum_{i=1}^{\infty}
\varepsilon_{i}\omega_{i}, \label{series}
\end{equation}
where each $\omega_{i}\in\mathcal{N}$ and $0\leq\varepsilon_{i}\leq
\varepsilon^{i}$.  To see this,\vspace*{1pt} express $\theta=\omega_{0}+r_{0}$, where
$\omega_{0}\in\mathcal{N}$ and $\Vert r_{0}\Vert _{2}\leq\varepsilon$. Then
express $\Vert r_{0}\Vert ^{-1}r_{0}\in S^{d-1}$ in a similar fashion, and iterate
this procedure.

Define the functional%
\[
\Vert x\Vert _{\mathcal{N}}=\max\bigl\{ \langle x,\omega \rangle\dvtx \omega
\in \mathcal{N}\bigr\}.
\]
As an easy consequence of the Cauchy--Schwarz inequality, provided $%
\varepsilon\in(0,1)$ we have%
%
\begin{equation}
(1-\varepsilon)\Vert x\Vert _{2}\leq\Vert x\Vert _{\mathcal{N}}\leq
\Vert x\Vert _{2}, \label{net aprox}
\end{equation}
which implies that%
%
\begin{equation}
B_{2}^{d}\subset B_{\mathcal{N}}\subset(1-
\varepsilon)^{-1}B_{2}^{d}, \label{geometric net
approx}
\end{equation}
where $B_{\mathcal{N}}=\{x\dvtx $ $\Vert x\Vert _{\mathcal{N}}\leq1\}$. The body
$B_{%
\mathcal{N}}$ is what remains if one deletes all open half-spaces that are
tangent to $B_{2}^{d}$ at points in $\mathcal{N}$.

A \textit{convex body} is a compact convex subset of Euclidean space with
nonempty interior. For a convex body $K\subset\mathbb{R}^{d}$ that contains
the origin as an interior point, its \textit{Minkowski functional} is
defined as%
\[
\Vert x\Vert _{K}=\inf\{\lambda>0\dvtx x\in\lambda K\}
\]
for all $x\in\mathbb{R}^{d}$.  By convexity of $K$, one can easily show
that $\Vert \cdot\Vert _{K}$ obeys the triangle inequality. The dual Minkowski
functional is defined as%
\[
\Vert y\Vert _{K^{\circ}}=\sup\bigl\{ \langle x,y \rangle\dvtx x\in K\bigr\}
\]
for all $y\in\mathbb{R}^{d}$, and the \textit{polar} of $K$ is%
\[
K^{\circ}=\bigl\{y\in\mathbb{R}^{d}\dvtx \Vert y\Vert
_{K^{\circ}}\leq1\bigr\}.
\]
By the Hahn--Banach theorem, $K^{\circ\circ}=K$.

The \textit{Hausdorff distance} $d_{\mathcal{H}}$ between $K$ and $L$ is
defined as%
\[
d_{\mathcal{H}}(K,L)=\max\Bigl\{\max_{k\in K}d(k,L);\max
_{l\in L}d(K,l)\Bigr\}.
\]
By convexity this reduces to%
\begin{eqnarray*}
d_{\mathcal{H}}(K,L) &=&\sup_{\theta\in S^{d-1}}\Bigl|\sup
_{k\in K} \langle k,\theta \rangle-\sup_{l\in L}
\langle l,\theta \rangle\Bigr|
\\
&=&\sup_{\theta\in S^{d-1}}\bigl|\bigl(\Vert \theta\Vert _{K^{\circ}}-\Vert
\theta \Vert _{L^{\circ
}}\bigr)\bigr|.
\end{eqnarray*}
We define the \textit{logarithmic Hausdorff distance} between $K$ and $L$
about a point $x\in \operatorname{int}(K\cap L)$ as%
\[
d_{\mathfrak{L}}(K,L,x)=\inf\bigl\{\lambda\geq1\dvtx \lambda
^{-1}(L-x)+x\subset K\subset\lambda(L-x)+x\bigr\}
\]
provided $\operatorname{int}(K\cap L)\neq\varnothing$, and
\[
d_{\mathfrak{L}}(K,L)=\inf\bigl\{d_{\mathfrak{L}}(K,L,x)\dvtx x\in \operatorname{int}(K\cap
L)\bigr\}.
\]
Note that%
\[
\log d_{\mathfrak{L}}(K,L,0)=\sup_{\theta\in S^{d-1}}\bigl|\log\Vert \theta
\Vert _{K}-\log\Vert \theta\Vert _{L}\bigr|.
\]
The following relations follow from the definitions above,%
%
\begin{eqnarray}
\label{af} d_{\mathfrak{L}}(K,L,0) &=&d_{\mathfrak{L}}\bigl(K^{\circ},L^{\circ},0
\bigr),
\nonumber
\\[-8pt]
\\[-8pt]
\nonumber
d_{\mathfrak{L}}(TK,TL) &=&d_{\mathfrak{L}}(K,L),
\end{eqnarray}
where $T$ is any invertible affine transformation. In addition, one can check
that%
%
\begin{eqnarray}
\label{HL} d_{\mathrm{BM}}(K,L) &\leq&d_{\mathfrak{L}}(K,L)^{2},
\nonumber
\\[-8pt]
\\[-8pt]
\nonumber
d_{\mathcal{H}}(K,L) &\leq&\operatorname{diam}(K) \bigl(d_{\mathfrak{L}}(K,L)-1\bigr);
\end{eqnarray}
hence all of our bounds in terms of $d_{\mathfrak{L}}$ apply equally
well to
$d_{\mathrm{BM}}$. For large bodies,\vadjust{\goodbreak} $d_{\mathcal{H}}$ is more sensitive than
$d_{%
\mathfrak{L}}$. More precisely, if $r\geq1$ and $rB_{2}^{d}+x\subset
K$ for
some $x\in\mathbb{R}^{d}$, and if $d_{\mathcal{H}}(K,L)\leq1/2$, then%
%
\begin{equation}
d_{\mathfrak{L}}(K,L)\leq1+2r^{-1}d_{\mathcal{H}}(K,L). \label{LH}
\end{equation}

By a simple compactness argument, there is an ellipsoid of maximal
volume $%
\mathcal{E}_{K}\subset K$. This ellipsoid is called the \textit{John
ellipsoid}~\cite{Ball2} associated to $K$. It can be shown that
$\mathcal{E}%
_{k}$ is unique and has the property that $K\subset d(\mathcal{E}_{k}-x)+x$,
where $x$ is the center of $\mathcal{E}_{k}$.  In particular,
$d_{\mathfrak{%
L}}(\mathcal{E}_{k},K)\leq d$.

In~\cite{Fr} it is shown that provided $\lambda<8^{-d}$, we have%
%
\begin{equation}
d_{\mathfrak{L}}(K,K_{\lambda},x)\leq1+8\lambda^{1/d}, \label{Fr
bound}
\end{equation}
where $x$ is the centroid of $K$ and $K_{\delta}$ is the floating body
inside $K$.

The \textit{cone measure} on $\partial K$ is defined as%
\[
\mu_{K}(E)=\operatorname{vol}_{d}\bigl(\bigl\{r\theta\dvtx
\theta\in E, r\in \lbrack 0,1]\bigr\}\bigr)
\]
for all measurable $E\subset\partial K$. The significance of the cone
measure is that it leads to a natural polar integration formula (see
\cite%
{NR}); for all $f\in L_{1}(\mathbb{R}^{d})$,
%
\begin{equation}
\int_{\mathbb{R}^{d}}f(x)\,dx=d\int_{0}^{\infty}
\int_{\partial
K}r^{d-1}f(r\theta)\,d\mu_{K}(
\theta)\,dr. \label{polar formula}
\end{equation}

A probability measure $\mu$ is called \textit{isotropic} if its centroid
lies at the origin and its covariance matrix is the $d\times d$ identity
matrix.

A function $f\dvtx \mathbb{R}^{d}\rightarrow\lbrack0,\infty)$ is called
\textit{log-concave} (see~\cite{KM}) if%
\[
f\bigl(\lambda x+(1-\lambda)y\bigr)\geq f(x)^{\lambda}f(y)^{1-\lambda}
\]
for all $x,y\in\mathbb{R}^{d}$ and all $\lambda\in(0,1)$. Any such
function can be written in the form $f(x)=e^{-g(x)}$ where $g\dvtx \mathbb{R}
^{d}\rightarrow(-\infty,\infty]$ is convex. If $f$ is the density of a
probability measure $\mu$, then it must decay exponentially to zero. In
this case $g$ lies above a cone, that is,%
%
\begin{equation}
g(x)\geq m\Vert x\Vert _{2}-c \label{domination of cone}
\end{equation}
with $m,c>0$. As a consequence of the Pr\'{e}kopa--Leindler inequality
\cite%
{Ball1}, if $x$ is a random vector with log-concave density, and $y$ is any
fixed vector, then $ \langle x,y \rangle$ has a log-concave
density in $\mathbb{R}$. Log-concave functions are very rigid. One such
example of this rigidity (see Lemma~5.12 in~\cite{LV}) is the fact that
if $%
\mathfrak{H}$ is any half-space containing the centroid of $\mu$, then
$\mu
(\mathfrak{H})\geq e^{-1}$. Another example (see Theorem~5.14 in~\cite{LV})
is that if $\mu$ is isotropic, then%
%
\begin{eqnarray}
2^{-7d}&\leq& f(0)\leq d(20d)^{d/2}, \label{rigid 1}
\\
(4\pi e)^{-d/2}&\leq&\Vert f\Vert _{\infty}\leq2^{8d}d^{d/2}
\label{rigid 3},
\end{eqnarray}
and if $\Vert x\Vert _{2}\leq1/9$, then%
%
\begin{equation}
2^{-8d}\leq f(x)\leq d2^{d}(20d)^{d/2} \label{rigid
2}.
\end{equation}
Let $1\leq p<\infty$. If $g\dvtx \mathbb{R}^{d}\rightarrow\lbrack
0,\infty]$
is convex and $\lim_{x\rightarrow\infty}g(x)=\infty$, then the
probability measure with density given by%
\[
f(x)=ce^{-g(x)^{p}}
\]
will be called $p$-log-concave. This is a natural generalization of the
normal distribution. If $f$ is $p$-log-concave, then it is also
$p^{\prime}$-log-concave for all $1\leq p^{\prime}\leq p$.

Let $\mathbf{H}_{d}$ denote the collection of all $(d-1)$-dimensional affine
subspaces (hyperplanes) of $\mathbb{R}^{d}$. The Radon transform of an
integrable
log-concave function $f\dvtx \mathbb{R}^{d}\rightarrow\lbrack0,\infty)$
is the
function $Rf\dvtx \mathbf{H}_{d}\rightarrow\lbrack0,\infty)$ defined by%
%
\begin{equation}
Rf(\mathcal{H})=\int_{\mathcal{H}}f(y)\,d_{\mathcal{H}}(y),
\label{Radon def}
\end{equation}
where $d_{\mathcal{H}}$ is Lebesgue measure on $\mathcal{H}$. The Radon
transform is closely related to the Fourier transform. See~\cite{Ko}
for a
discussion of these operators and their connections to convex geometry.

\section{The one-dimensional case}

Let $f$ be a nonvanishing log-concave probability density function on $
\mathbb{R}$ associated to a probability measure $\mu$.  In
particular, $%
f(t)=e^{-g(t)}$ where $g\dvtx \mathbb{R}\rightarrow\mathbb{R}$ is convex.  For $%
t\in\mathbb{R}$, define%
\begin{eqnarray*}
J(t) &=&\int_{-\infty}^{t}f(s)\,ds,
\\
u(t) &=&-\log\bigl(1-J(t)\bigr).
\end{eqnarray*}
The cumulative distribution function $J$ is a strictly increasing bijection
between $\mathbb{R}$ and $(0,1)$.  The following lemma is a standard
result; see, for example, Theorem~5.1 in~\cite{LV} for the statement,
and the references
given there.  However, we include a short proof here for completeness.

\begin{lemma}
$u$ is convex.
\end{lemma}

\begin{pf}
Assume momentarily that $g\in C^{2}(\mathbb{R})$.  For $t\in(0,1)$
define%
\[
\psi(t)=f\bigl(J^{-1}(1-t)\bigr).
\]
Note that%
\[
\psi^{\prime\prime}(t)=\frac{-g^{\prime\prime}(J^{-1}(1-t))}{\psi
(t)}%
\leq0.
\]
Hence $\psi$ is concave.  In addition, $\lim_{t\rightarrow0}\psi
(t)=\lim_{t\rightarrow1}\psi(t)=0$.  Hence, the function $\kappa
(t)=\psi
(t)/t$ is nonincreasing on $(0,1)$, and the function
$f(t)/(1-J(t))=\kappa
(1-J(t))$ is nondecreasing on $\mathbb{R}$.  Since $u^{\prime
}(t)=f(t)/(1-J(t))$, $u$ is convex.

If $g\notin C^{2}(\mathbb{R})$, then the result follows by approximation
(convolve $\mu$ with a Gaussian).
\end{pf}

\begin{lemma}
\label{little inequality}If $T\geq1$ and $x>2T\log T$, then $(\log
x)/x<T^{-1}$.\vadjust{\goodbreak}
\end{lemma}

\begin{pf}
Since the function $y=e^{-1}x$ is tangent to the strictly concave
function $%
y=\log x$, the function $y=(\log x)/x$ has a global maximum of $e^{-1}$ and
is decreasing on $[e,\infty)$. We now consider two cases. In case 1, $T<e$
and therefore $(\log x)/x\leq e^{-1}<T^{-1}$. In case 2, $T\geq e$.
Since $%
(\log T)/T<2^{-1}$, it follows that $\log(2\log T)<\log T$. For
$x^{\prime
}=2T\log T$,%
\[
\frac{\log x^{\prime}}{x^{\prime}}=\frac{\log T+\log(2\log
T)}{2T\log T}<%
\frac{1}{T}.
\]
Since $x>x^{\prime}>e$, $(\log x)/x<(\log x^{\prime})/x^{\prime}$
and the
result follows.
\end{pf}

The following lemma is a quantitative version of the Gnedenko law of large
numbers for log-concave probability measures on $\mathbb{R}$.

\begin{lemma}
\label{1 D Gnedenko}Let $q\geq1$ and $n\geq120q^{2}(2+\log q)^{2}$.
Let $%
\mu$ be a probability measure on $\mathbb{R}$ with a nonvanishing
log-concave density function and cumulative distribution function $J$, and
let $(\gamma_{i})_{1}^{n}$ be an i.i.d. sample from $\mu$. With
probability at least $1-2(\log n)^{-q}$,%
%
\begin{equation}
\frac{\llvert \gamma_{(n)}-J^{-1}(1-1/n)\rrvert }{J^{-1}(1-1/n)-
\mathbb{E}\mu}\leq6q\frac{\log\log n}{\log n}, \label{log concave gnedenko}
\end{equation}
where $\gamma_{(n)}=\max\{\gamma_{i}\}_{1}^{n}$ and $\mathbb{E}\mu$
denotes the centroid of $\mu$.
\end{lemma}

\begin{pf}
We shall implicitly make use of Lemma~\ref{little inequality} several times
throughout the proof. Let $a=(\log n)^{-q}$ and $b=q\log n$. It follows that
$0<a<b<ne^{-1}$. Set $s=J^{-1}(1-b/n)$ and $t=J^{-1}(1-a/n)$. As mentioned
in the preliminaries (see also Lemma 3.3 in~\cite{Bo}), $1-J(\mathbb
{E}\mu
)\geq e^{-1}$, hence $u(\mathbb{E}\mu)\leq1$.  Since $b/n<e^{-1}$, we
have $\mathbb{E}\mu<s<t$. By convexity of $u$ we have the inequality
$(s-%
\mathbb{E}\mu)^{-1}(u(s)-u(\mathbb{E}\mu))\leq(t-s)^{-1}(u(t)-u(s))$
which can be rewritten as%
%
\begin{equation}
\frac{J^{-1}(1-a/n)-J^{-1}(1-b/n)}{J^{-1}(1-b/n)-\mathbb{E}\mu}\leq \frac{%
\log b-\log a}{\log n-\log b-1} \label{1 d bound}.
\end{equation}
Since $2qe\log\sqrt{n}\leq\sqrt{n}$, it follows that $\log(qe\log
n)\leq
\frac{1}{2}\log n$ which implies that%
\[
\frac{\log b-\log a}{\log n-\log b-1}\leq\frac{3q\log\log n}{\log
n-\log
(qe\log n)}\leq6q\frac{\log\log n}{\log n}.
\]
By independence,%
\begin{eqnarray*}
&&\mathbb{P}\bigl\{J^{-1}(1-b/n) \leq\gamma_{(n)}\leq
J^{-1}(1-a/n)\bigr\}
\\
&&\qquad= \biggl( 1-\frac{a}{n} \biggr) ^{n}- \biggl( 1-
\frac{b}{n} \biggr) ^{n}
\\
&&\qquad\geq 1-a-e^{-b}
\\
&&\qquad \geq 1-2(\log n)^{-q}.
\end{eqnarray*}
If the event $\{J^{-1}(1-b/n)\leq\gamma_{(n)}\leq J^{-1}(1-a/n)\}$ occurs,
then the event defined by inequality (\ref{log concave gnedenko}) also
occurs.
\end{pf}

Although Lemma~\ref{1 D Gnedenko} applies to the multiplicative version of
the Gnedenko law of large numbers, it also recovers the additive
version as
long as%
%
\begin{equation}
J^{-1}(1-1/n)=o \biggl( \frac{\log n}{\log\log n} \biggr). \label{needed}
\end{equation}
If, in the proof, we take $a^{-1}=b=\log_{(m)}n$ (the $m${th} iterate of
the logarithm), then the probability bound becomes $1-2(\log_{(m)}n)^{-1}$,
and the right-hand side of (\ref{log concave gnedenko}) becomes%
\[
\frac{4\log_{(m+1)}n}{\log n}
\]
provided $n>n_{0}(m)$.

\section{Main proofs}
Since Lebesgue measure depends on the underlying Euclidean structure of
$%
\mathbb{R}^{d}$, so does the definition of $f=d\mu/dx$, and therefore also
the definition of $D_{\delta}=\operatorname{Cl}\{x\dvtx f(x)\geq\delta\}$. A natural
variation of the body $D_{\delta}$ which does not depend on Euclidean
structure is the body%
\[
D_{\delta}^{\natural}=\operatorname{Cl}\bigl\{x\in\mathbb{R}^{d}\dvtx f(x)
\geq\tau _{d}^{-1}9^{d}\bigl|\det \operatorname{cov}(\mu)\bigr|^{-1/2}
\delta\bigr\},
\]
where the quantity%
%
\begin{equation}
\tau_{d}=\operatorname{vol}_{d-1}\bigl(B_{2}^{d-1}
\bigr)\int_{1/2}^{1}\bigl(1-t^{2}
\bigr)^{(d-1)/2}\,dt \label{tau def}
\end{equation}
represents the volume of the set $\{x\in\mathbb{R}^{d}\dvtx \Vert x\Vert _{2}\leq
1$, $%
x_{1}\geq1/2\}$. Associated to $D_{\delta}^{\natural}$ are three
ellipsoids that play a central role in our proof. The John ellipsoid of
$%
D_{\delta}^{\natural}$ is denoted $\mathcal{E}_{D_{\delta}^{\natural}}$
and the centroid of $\mathcal{E}_{D_{\delta}^{\natural}}$ will be
denoted $%
\mathcal{O}_{\delta}$. We also consider%
%
\begin{equation}
\mathcal{E}_{\delta}^{\sharp}=3d(\mathcal{E}_{D_{\delta}^{\natural
}}-%
\mathcal{O}_{\delta})+\mathcal{O}_{\delta} \label{sharp ellipsoid}
\end{equation}
and%
%
\begin{equation}
\mathcal{E}_{\delta}^{\flat}=\tfrac{1}{2}(
\mathcal{E}_{D_{\delta
}^{\natural}}-\mathcal{O}_{\delta})+\mathcal{O}_{\delta}
\label{flat ellipsoid}.
\end{equation}
The advantage of using $D_{\delta}^{\natural}$ is that we may place
$\mu$
in different positions at various stages of our analysis. We first
position $%
\mu$ to be isotropic and then position it so that $\mathcal
{E}_{D_{\delta
}^{\natural}}=B_{2}^{d}$. We include the proofs of Lemmas~\ref{basic
positioning} and~\ref{epsilon net positioning} in Section~\ref{tech}.

\begin{lemma}
\label{basic positioning}There exists a universal constant $c>0$ with the
following property. Let $d\in\mathbb{N}$, let $\mu$ be a log-concave
probability measure with a continuous density function $f$, and let
$\delta
<c\exp(-8d^{2}\log d)$. Let $\mathfrak{H}$ be a half-space (either\vadjust{\goodbreak}
open or
closed) with $\mu(\mathfrak{H})=\delta$, and let $\mathcal{E}_{\delta
}^{\sharp}$ and $\mathcal{E}_{\delta}^{\flat}$ be defined by (\ref{sharp
ellipsoid}) and (\ref{flat ellipsoid}), respectively. Then%
%
\begin{eqnarray}
\mathfrak{H}\cap\mathcal{E}_{\delta}^{\sharp} &\neq&\varnothing,
\label{upper bound for F}
\\
\mathfrak{H}\cap\mathcal{E}_{\delta}^{\flat} &=&\varnothing
\label{lower bound for F}.
\end{eqnarray}
Consequently,%
\[
\mathcal{E}_{\delta}^{\flat}\subset F_{\delta}\subset
\mathcal {E}_{\delta
}^{\sharp}.
\]
\end{lemma}

We shall use the Euclidean structure corresponding to $\mathcal{E}%
_{D_{\delta}^{\natural}}$ in order to compare $F_{1/n}$ and $P_{n}$. The
following lemma together with Lemma~\ref{basic positioning} allows us
to do
so.

\begin{lemma}
\label{epsilon net positioning}Let $d\in\mathbb{N}$ and let $K$ and
$L$ be
convex bodies in $\mathbb{R}^{d}$ such that $0\in \operatorname{int}(L)$ and $%
rB_{2}^{d}\subset K\subset RB_{2}^{d}$ for some $r,R>0$. Let $0<\rho<1/2$
and $0<\varepsilon<(16R/r)^{-1}$, and let $\mathcal{N}$ be an
$\varepsilon$%
-net in $S^{d-1}$. Suppose that for each $\omega\in\mathcal{N}$,%
%
\begin{equation}
(1-\rho)\Vert \omega\Vert _{L}\leq\Vert \omega\Vert _{K}
\leq(1+\rho)\Vert \omega\Vert _{L} \label{tight KL bound in net}.
\end{equation}
Then for all $x\in\mathbb{R}^{d}$ we have%
%
\begin{equation}
\bigl(1+2\rho+28Rr^{-1}\varepsilon\bigr)^{-1}\Vert x\Vert
_{L}\leq\Vert x\Vert _{K}\leq \bigl(1+2\rho
+28Rr^{-1}\varepsilon\bigr)\Vert x\Vert _{L}. \label{norm
approx}
\end{equation}
In particular,%
%
\begin{equation}
d_{\mathfrak{L}}(K,L)\leq d_{\mathfrak{L}}(K,L,0)\leq1+2\rho
+28Rr^{-1}\varepsilon. \label{geometric distance bound}
\end{equation}
\end{lemma}

\begin{pf*}{Proof of Theorem~\ref{main 1}}
By convolving $\mu$ with a Gaussian measure of the form%
\[
\phi_{\lambda,d}(x)=\lambda^{-d}\phi_{d}\bigl(
\lambda^{-1}x\bigr),
\]
where $\phi_{d}(x)=(2\pi)^{-d}\exp(-2^{-1}\Vert x\Vert _{2}^{2})$ is the standard
normal density function, and taking $\lambda\rightarrow0$, we may assume
that the density of $\mu$ is continuous and nonvanishing. This is possible
because the bounds in the theorem do not depend on $\mu$. The
condition $%
n\geq c\exp\exp(5d)+c^{\prime}\exp(q^{3})$ (with sufficiently large $c$
and $c^{\prime}$) insures that the probability bound is nontrivial.
It is
also sufficiently large so that we may use Lemma~\ref{basic positioning}
with $\delta=1/n$ and Lemma~\ref{1 D Gnedenko} with $q^{\prime}=d+q$. In
fact we will implicitly make use of this bound repeatedly throughout the
proof. Let $\varepsilon=(\log n)^{-1}$.  By applying a suitable affine
transformation, we may assume that $\mathcal{E}_{D_{1/n}^{\natural
}}=B_{2}^{d}$. By Lemma~\ref{basic positioning}, if $\mathfrak
{H}_{\theta,t} $ is a half-space with $\mu(\mathfrak{H}_{\theta,t})=1/n$, then%
%
\begin{equation}
1/2\leq t\leq3d. \label{depth}
\end{equation}
This implies that $1/2B_{2}^{d}\subset F_{1/n}\subset3 dB_{2}^{d}$. For each
$\theta\in S^{d-1}$, the function $f_{\theta}(t)=-\frac{d}{dt}\mu(%
\mathfrak{H}_{\theta,t})$ is the density of a log-concave probability
measure $\mu_{\theta}$ on $\mathbb{R}$ with cumulative distribution
function $J_{\theta}(t)=1-\mu(\mathfrak{H}_{\theta,t})$.  Furthermore,
the sequence $( \langle\theta,x_{i} \rangle)_{i=1}^{n}$ is an
i.i.d. sample from this distribution.\vadjust{\goodbreak}  Recalling the definition of the dual
Minkowski functional, for any $y\in\mathbb{R}^{d}$,%
\begin{eqnarray*}
\Vert y\Vert _{P_{n}^{\circ}} &=&\sup\bigl\{ \langle x,y \rangle\dvtx x\in
P_{n}\bigr\}
\\
&=&\max_{i=1,\ldots, n} \langle x_{i},y \rangle.
\end{eqnarray*}
We use this notation even when $0\notin \operatorname{int}(P_{n})$.  Let $\mathcal{N}$
denote a generic $\varepsilon$-net in $S^{d-1}$, and consider the
function%
\[
\widetilde{f}_{\mathcal{N}}(x)=\inf\bigl\{\mu(\mathfrak{H}_{\omega,t})
\dvtx \omega \in\mathcal{N}, t= \langle\omega,x \rangle\bigr\}.
\]
For all $\delta>0$, define the discrete floating body%
\[
F_{\delta}^{\mathcal{N}}=\bigl\{x\in\mathbb{R}^{d}\dvtx
\widetilde{f}_{\mathcal
{N}%
}(x)\geq\delta\bigr\}.
\]
Note that $\widetilde{f}(x)=\inf_{\mathcal{N}}\widetilde{f}_{\mathcal{N}}(x)$
and $F_{\delta}=\bigcap_{\mathcal{N}}F_{\delta}^{\mathcal{N}}$, where $%
\mathcal{N}$ runs through the collection of all $\varepsilon$-nets in $
S^{d-1}$.  By (\ref{depth}), $\frac{1}{2}B_{2}^{d}\subset
F_{1/n}^{\mathcal{%
N}}\subset3 dB_{\mathcal{N}}$, and by (\ref{geometric net approx}) we
have $%
1/2B_{2}^{d}\subset F_{1/n}^{\mathcal{N}}\subset4 dB_{2}^{d}$ which implies
that $(4d)^{-1}B_{2}^{d}\subset(F_{1/n}^{\mathcal{N}})^{\circ}\subset
2B_{2}^{d}$.  For each $\theta\in S^{d-1}$, we have%
\begin{eqnarray*}
\mathbb{E}\mu_{\theta} &\geq&J_{\theta}^{-1}
\bigl(e^{-1}\bigr)
\\
&\geq&J_{\theta}^{-1}(1/n).
\end{eqnarray*}
Combining this and (\ref{log concave gnedenko}), with probability at
least $%
1-2(\log n)^{-d-q}$ we have that%
\[
\frac{\llvert \Vert \theta\Vert _{P_{n}^{\circ}}-J_{\theta
}^{-1}(1-1/n)\rrvert }{J_{\theta}^{-1}(1-1/n)-J_{\theta
}^{-1}(1/n)}%
\leq6(d+q)\frac{\log\log n}{\log n}.
\]
Since both $-J_{\theta}^{-1}(1/n)$ and $J_{\theta}^{-1}(1-1/n)$ lie
in the
interval $[1/2,3d]$, both have roughly the same order of magnitude, and we
have
\[
(1-\rho)J_{\theta}^{-1}(1-1/n)\leq\Vert \theta\Vert
_{P_{n}^{\circ}}\leq (1+\rho)J_{\theta}^{-1}(1-1/n),
\]
where
\[
\rho=42d(d+q)\frac{\log\log n}{\log n}<1/8.
\]
With probability at least $1-\varepsilon^{-d}3^{d+1}(\log
n)^{-d-q}=1-3^{d+1}(\log n)^{-q}$, this holds for all $\omega\in
\mathcal{N}
$.  Hence,
\[
(1+\rho)^{-1}P_{n}\subset F_{1/n}^{\mathcal{N}},
\]
which implies that%
\[
(1-\rho)\Vert \theta\Vert _{P_{n}^{\circ}}\leq\Vert \theta \Vert
_{(F_{1/n}^{\mathcal{N}%
})^{\circ}}
\]
for all $\theta\in S^{d-1}$. On the other hand, for all $\omega\in
\mathcal{N}$ we have%
%
\begin{eqnarray}\label{noname q eqn}
\Vert \omega\Vert _{P_{n}^{\circ}} &\geq&(1-\rho)J_{\omega}^{-1}(1-1/n)
\nonumber
\\[-8pt]
\\[-8pt]
\nonumber
&\geq&(1-\rho)\Vert \omega\Vert _{(F_{1/n}^{\mathcal{N}})^{\circ}}
\end{eqnarray}
As $\Vert \cdot\Vert _{P_{n}^{\circ}}$ is the supremum of Lipschitz
functions, $%
\operatorname{Lip}(\Vert \cdot\Vert _{P_{n}^{\circ}})\leq\sup\{\Vert x\Vert _{2}\dvtx\break x\in P_{n}\}<8d$.
Since (\ref{noname q eqn}) implies that $\Vert \omega\Vert _{P_{n}^{\circ}}>1/4$
(simultaneously for all $\omega\in\mathcal{N}$ with high
probability), for
all $\theta\in S^{d-1}$ we have $\Vert \theta\Vert _{P_{n}^{\circ}}>1/8$. Using
the Hahn--Banach theorem, $0\in \operatorname{int}(P_{n})$. This also implies that
$0\in
\operatorname{int}(P_{n}^{\circ})$. By~(\ref{norm approx}),%
%
\begin{equation}
(1+4\rho+224 d\varepsilon)^{-1}\Vert x\Vert _{P_{n}^{\circ}}\leq \Vert x
\Vert _{(F_{1/n}^{%
\mathcal{N}})^{\circ}}\leq(1+4\rho+224 d\varepsilon )\Vert x\Vert _{P_{n}^{\circ}}
\label{first net approx}
\end{equation}
for all $x\in\mathbb{R}^{d}$.  Let $\mathcal{M}$ be any other
$\varepsilon
$-net in $S^{d-1}$. By the calculations above, with probability at
least $%
1-3^{d+1}(\log n)^{-q}$,%
%
\begin{equation}
(1+4\rho+224 d\varepsilon)^{-1}\Vert x\Vert _{P_{n}^{\circ}}\leq \Vert x
\Vert _{(F_{1/n}^{%
\mathcal{M}})^{\circ}}\leq(1+4\rho+224 d\varepsilon )\Vert x\Vert _{P_{n}^{\circ}}
\label{second net approx}
\end{equation}
for all $x\in\mathbb{R}^{d}$. By the union bound, with probability at least
$1-\break3^{d+2}(\log n)^{-q}>0$, both (\ref{first net approx}) and (\ref{second
net approx}) hold, which implies that%
%
\begin{equation}
(1+4\rho+224 d\varepsilon)^{-2}F_{1/n}^{\mathcal{N}}\subset
F_{1/n}^{%
\mathcal{M}}\subset(1+4\rho+224 d\varepsilon)^{2}F_{1/n}^{\mathcal{N}}.
\label{third net approx}
\end{equation}
However both $F_{1/n}^{\mathcal{N}}$ and $F_{1/n}^{\mathcal{M}}$ are
deterministic bodies, and (\ref{third net approx}) therefore holds
unconditionally. Since $F_{1/n}=\bigcap_{\mathcal{M}}F_{1/n}^{\mathcal{M}}$,
where the intersection is taken over all $\varepsilon$-nets in $S^{d-1}$,
we have%
\[
(1+4\rho+224 d\varepsilon)^{-2}F_{1/n}^{\mathcal{N}}\subset
F_{1/n}\subset (1+4\rho+224 d\varepsilon)^{2}F_{1/n}^{\mathcal{N}}.
\]
Combining this with the polar of (\ref{first net approx}) gives that with
probability at least $1-3^{d+1}(\log n)^{-q}$, we have%
\[
(1+4\rho+224 d\varepsilon)^{-3}P_{n}\subset F_{1/n}
\subset(1+4\rho +224 d\varepsilon)^{3}P_{n}
\]
from which the result follows by the inequality $(1+\varepsilon^{\prime
})^{3}\leq1+12\varepsilon^{\prime}$, valid if $0\leq\varepsilon
^{\prime
}\leq1$.
\end{pf*}

\begin{lemma}
Let $g\dvtx \mathbb{R}^{d}\rightarrow\lbrack0,\infty]$ be convex with $%
\lim_{x\rightarrow\infty}g(x)=\infty$, let $K\subset\mathbb{R}^{d}$
be a
convex body containing $0$ in its interior and let $p>1$.  Then there
exist $c_{1},c_{2}>0$ such that for all $x\in\mathbb{R}^{d}$,
%
\begin{equation}
g(x)^{p}\geq c_{1}\Vert x\Vert _{K}^{p}-c_{2}.
\label{domination}
\end{equation}
\end{lemma}

\begin{pf}
We leave the easy proof of this to the reader.
\end{pf}

\begin{lemma}
\label{lem tail decay}Let~$p>1$, $d\in\mathbb{N}$ and let $\mu$ be a
$p$%
-log-concave probability measure on $\mathbb{R}^{d}$.  Then there
exist $%
c_{1},c_{2},t_{0}>0$ such that for\vadjust{\goodbreak} all $\theta\in S^{d-1}$ and all
$t\geq
t_{0}$,%
%
\begin{equation}
\mu(\mathfrak{H}_{\theta,t})\leq c_{1}t^{1-p}e^{-c_{2}t^{p}},
\label{tail decay}
\end{equation}
where $\mathfrak{H}_{\theta,t}=\{x\in\mathbb{R}^{d}\dvtx  \langle
x,\theta
\rangle\geq t\}$.
\end{lemma}

\begin{pf}
For all $t\geq1$ we have%
\begin{eqnarray*}
e^{-t^{p}} &\leq&-\frac{d}{dt} \bigl( p^{-1}t^{1-p}e^{-t^{p}}
\bigr)
\\
&=&p^{-1}(p-1)t^{-p}e^{-t^{p}}+e^{-t^{p}}
\\
&\leq&p^{-1}(2p-1)e^{-t^{p}}.
\end{eqnarray*}
Hence, by the fundamental theorem of calculus,%
%
\begin{equation}
(2p-1)^{-1}t^{1-p}e^{-t^{p}}\leq\int
_{t}^{\infty}e^{-s^{p}}\,ds\leq
p^{-1}t^{1-p}e^{-t^{p}}. \label{sandwitch}
\end{equation}
Since the image of a $p$-log-concave probability measure under an orthogonal
transformation is $p$-log-concave, we may assume without loss of generality
that $\theta=e_{1}=(1,0,0,0,\ldots)$. By (\ref{domination}), there
exist $%
c_{1},c_{2}>0$ such that for all $x\in\mathbb{R}^{d}$,
\[
f(x)\leq c_{1}e^{-c_{2}\Vert x\Vert _{p}^{p}},
\]
where $\Vert x\Vert _{p}^{p}=\sum_{i=1}^{d}|x_{i}|^{p}$. Hence,%
%
\begin{eqnarray}\label{tail bound}
\mu(\mathfrak{H}_{\theta,t}) &\leq&\int_{\mathfrak{H}_{\theta,t}}c_{1}e^{-c_{2}\Vert x\Vert _{p}^{p}}\,dx
\nonumber
\\[-8pt]
\\[-8pt]
\nonumber
&=&\int_{t}^{\infty}c_{3}e^{-c_{2}s^{p}}\,ds.
\end{eqnarray}
The result now follows from a change of variables, (\ref{tail bound})
and (%
\ref{sandwitch}).
\end{pf}

\begin{pf*}{Proof of Theorem~\ref{main 2}}
Let $c_{1}$, $c_{2}$ and $t_{0}$ be the constants appearing in Lemma~\ref%
{lem tail decay}. Let $n_{0}>c_{1}+\exp(2^{-1}c_{2}t_{0}^{p})$. Without
loss of generality, $t_{0}>1$ and $n>n_{0}$. Set $\alpha
=(2c_{2}^{-1}\log
n)^{1/p}$ and consider any $x\in\mathbb{R}^{d}$ with $\Vert x\Vert _{2}>\alpha$.
Let $\theta=\Vert x\Vert _{2}^{-1}x$ and $t=(\alpha+\Vert x\Vert _{2})/2$. Since
$t>\alpha
>t_{0}$ and $n>c_{1}$, Lemma~\ref{lem tail decay} implies that%
\[
\mu(\mathfrak{H}_{\theta,t})<c_{1}n^{-2}<n^{-1}.
\]
Since $\Vert x\Vert _{2}>t$, $x\in \operatorname{int}(\mathfrak{H}_{\theta,t})$. By
definition of
the floating body, $x\notin F_{1/n}$. Since this is true for all such
$x$, $%
\operatorname{diam}(F_{1/n})\leq2\alpha=c_{4}(\log n)^{1/p}$. The result now follows from
Theorem~\ref{main 1} and relation (\ref{HL}) between the Hausdorff and
the logarithmic Hausdorff distances.
\end{pf*}

\section{Technical lemmas}\label{tech}

This section contains some technical results on the rigidity of log-concave
functions that enable us to obtain a lower bound on the sample size.\vadjust{\goodbreak}

\begin{lemma}
\label{vol ball}There exist universal constants $c_{1},c_{2}>0$ such that
for all $d\in\mathbb{N}$,%
\[
c_{1}^{d}d^{-d/2}\leq\operatorname{vol}_{d}
\bigl(B_{2}^{d}\bigr)\leq c_{2}^{d}d^{-d/2}.
\]
\end{lemma}

\begin{pf}
This follows from Stirling's formula and the expression $\operatorname{vol}%
_{d}(B_{2}^{d})=\pi^{d/2}(\Gamma(1+d/2))^{-1}$; see Corollary 2.20 in
\cite%
{Ko} or page 11 in~\cite{Pisier}.
\end{pf}

\begin{lemma}
\label{density decay}There exists a universal constant $c>0$ with the
following property. Let $d\in\mathbb{N}$ and let $\mu$ be an isotropic
log-concave probability measure on $\mathbb{R}^{d}$ with density
function $%
f $. For all $x\in\mathbb{R}^{d}$,
\[
f(x)\leq e^{-\alpha_{d}\Vert x\Vert _{2}+\beta_{d}},
\]
where $\alpha_{d}=c^{d}d^{-d/2}$ and $\beta_{d}=10d\log(d)+7$.
\end{lemma}

\begin{pf}
We first consider the case $d\geq2$. The volume of a cone in $\mathbb
{R}%
^{d} $ with height $h$ and base radius $r$ is $d^{-1}r^{d-1}h\mathrm
{vol}%
_{d-1}(B_{2}^{d-1})$. For any $x\in\mathbb{R}^{d}$, let $A_{x}$ be the cone
with vertex $x$ and base $(1/9)B_{2}^{d}\cap x^{\bot}$. Then $\mathrm
{vol}%
_{d}(A_{x})=d^{-1}9^{-d+1}\Vert x\Vert _{2}\operatorname{vol}%
_{d-1}(B_{2}^{d-1})>e^{-4d+3}\Vert x\Vert _{2}\operatorname{vol}_{d-1}(B_{2}^{d-1})$. By
log-concavity of $f$ and inequality (\ref{rigid 2}), for all $y\in
A_{x}$,%
%
\begin{equation}
f(y)\geq\min\bigl\{f(x),2^{-8d}\bigr\} \label{mango}.
\end{equation}
If $f(x)\geq2^{-8d}$, then%
\[
1\geq\int_{A_{x}}f(y)\,dy\geq2^{-8d}
\operatorname{vol}%
_{d}(A_{x})>e^{-10d+3}\Vert x
\Vert _{2}\operatorname{vol}_{d-1}\bigl(B_{2}^{d-1}
\bigr),
\]
and it follows that%
\[
\Vert x\Vert _{2}<\frac{e^{10d-3}}{\operatorname{vol}_{d-1}(B_{2}^{d-1})}.
\]
Hence, if $\Vert x\Vert _{2}\geq e^{10d-3}/\operatorname{vol}_{d-1}(B_{2}^{d-1})$,
then%
%
\begin{equation}
f(x)<2^{-8d} \label{f small}.
\end{equation}
For any such $x$ we have the convex combination%
\[
\frac{e^{10d-3}}{\operatorname{vol}_{d-1}(B_{2}^{d-1})}\frac
{x}{\Vert x\Vert _{2}}=\frac{%
e^{10d-3}}{\Vert x\Vert _{2}\operatorname{vol}_{d-1}(B_{2}^{d-1})}x+ \biggl( 1-
\frac{%
e^{10d-3}}{\Vert x\Vert _{2}\operatorname{vol}_{d-1}(B_{2}^{d-1})} \biggr) 0.
\]
Set%
\[
\widetilde{x}=\frac{e^{10d-3}}{\operatorname{vol}_{d-1}(B_{2}^{d-1})}\frac{x}{
\Vert x\Vert _{2}}.
\]
Using concavity of $\log f$ and inequality (\ref{f small}),%
\[
-8d\ln2\geq \biggl( \frac{e^{10d-3}}{\Vert x\Vert _{2}\mathrm
{vol}_{d-1}(B_{2}^{d-1})%
} \biggr) \log f(x)+ \biggl( 1-
\frac{e^{10d-3}}{\Vert x\Vert _{2}\operatorname{vol}%
_{d-1}(B_{2}^{d-1})} \biggr) \log f(0).\vadjust{\goodbreak}
\]
After some simplification, and using inequality (\ref{rigid 1}), we get%
\[
f(x)\leq\exp \bigl( -de^{-10d+3}\mathrm {vol}_{d-1}
\bigl(B_{2}^{d-1}\bigr)\Vert x\Vert _{2}\ln 2-7d
\ln2 \bigr).
\]
If, on the other hand, $\Vert x\Vert _{2}<e^{10d-3}/\mathrm
{vol}_{d-1}(B_{2}^{d-1})$%
, then by (\ref{rigid 3})%
\begin{eqnarray*}
f(x) &\leq&\Vert f\Vert _{\infty}\leq d^{d/2}2^{8d}
\\
&=&\exp\bigl(2^{-1}d\ln d+8d\ln2\bigr).
\end{eqnarray*}
Using Lemma~\ref{vol ball}, it follows that for all $x\in\mathbb
{R}^{d}$,%
\begin{eqnarray*}
f(x) &\leq&\exp \bigl( -de^{-10d+3}\operatorname{vol}_{d-1}
\bigl(B_{2}^{d-1}\bigr)\ln 2\Vert x\Vert _{2}+9d
\ln2+2^{-1}d\ln d \bigr)
\\
&\leq&\exp\bigl(-c_{3}^{d}d^{-d/2}\Vert x\Vert
_{2}+10d\ln d\bigr).
\end{eqnarray*}
The case $d=1$ is simpler, and we leave the details to the reader.
First show
that $f(2^{8})\leq2^{-8}$, and then proceed as in the case $d\geq2$ to
obtain $f(x)\leq\exp(-2^{-9}|x|+7)$ for all $x\in\mathbb{R}$.
\end{pf}

\begin{corollary}
\label{vol growth}There exist universal constants $c_{1},c_{2}>0$ with the
following property. Let $d\in\mathbb{N}$, and let $\mu$ be an absolutely
continuous isotropic log-concave probability measure. For all $\delta
<e^{-10d\log d-7}$,%
\[
D_{\delta}\subset c_{1}^{d}d^{d/2}\bigl(
\log\delta^{-1}\bigr)B_{2}^{d}.
\]
In particular, $\operatorname{vol}_{d}(D_{\delta})\leq c_{2}\exp(d^{2}\log
d)(\log\delta^{-1})^{d}$.
\end{corollary}

\begin{pf}
By (\ref{rigid 1}), $D_{\delta}\neq\varnothing$. By the bounds on
$\delta$, it follows that $10d\log d+7\leq\log\delta^{-1}$. The result now
follows from Lemmas~\ref{density decay} and~\ref{vol ball}.
\end{pf}

\begin{lemma}
\label{UDR}There exists a universal constant $c>0$ with the following
property. Let $d\in\mathbb{N}$, and let $\mu$ be an isotropic log-concave
probability measure with density $f$. Let $r>1$ and $x\in\mathbb{R}^{d}$.
If $f(x)<2^{-8d}$, then%
%
\begin{equation}
f(rx)\leq f(x)\exp\bigl(-c^{d}d^{-d/2}(r-1)\Vert x\Vert
_{2}\bigr). \label{uniform decay rate}
\end{equation}
\end{lemma}

\begin{pf}
Let $g=-\log f$. By Lemmas~\ref{density decay} and~\ref{vol ball},
there exists a universal constant $c_{2}>0$ such that $f(\widetilde
{x})\leq
2^{-8d}$ for all $\widetilde{x} $ with $\Vert \widetilde
{x}\Vert _{2}\geq
c_{2}^{d}d^{d/2}$; see in particular~(\ref{f small}). Let $x\in\mathbb
{R}%
^{d}$ be the point specified in the statement of the lemma. We consider two
cases. In the first case $\Vert x\Vert _{2}\geq c_{2}^{d}d^{d/2}$. Let
$\widetilde{x}%
=c_{2}^{d}d^{d/2}\Vert x\Vert _{2}^{-1}x$. By inequality (\ref{rigid 1}),
$f(0)\geq
2^{-7d}$. By convexity of $g$ and the definition of $c_{2}$,%
\[
\frac{g(rx)-g(x)}{(r-1)\Vert x\Vert _{2}}\geq\frac{g(\widetilde{x})-g(0)}{\Vert %
\widetilde{x}\Vert _{2}}=\Vert \widetilde{x}\Vert
_{2}^{-1}\ln\frac
{f(0)}{f(\widetilde{x%
})}\geq c_{2}^{-d}d^{1-d/2}
\ln2.
\]
In the second case, $\Vert x\Vert _{2}<c_{2}^{d}d^{d/2}$. Recall that, by
hypothesis, $f(x)<2^{-8d}$. Therefore,%
\[
\frac{g(rx)-g(x)}{(r-1)\Vert x\Vert _{2}}\geq\frac{g(x)-g(0)}{\Vert x\Vert _{2}}\geq \Vert x\Vert _{2}^{-1}
\ln\frac{f(0)}{f(x)}\geq c_{2}^{-d}d^{1-d/2}\ln(2)
\]
from which the result follows with $c=(2c_{2})^{-1}$.\vadjust{\goodbreak}
\end{pf}

\begin{lemma}
\label{mass outside contour}There exists a universal constant $c_{1}>0$ with
the following property. Let $d\in\mathbb{N}$, and let $\mu$ be an isotropic
log-concave probability measure with a continuous density function $f$. For
all $\delta<e^{-10d\log d-7}$,%
%
\begin{equation}
\mu\bigl(\mathbb{R}^{d}\setminus D_{\delta}\bigr)\leq
\alpha_{d}\delta\bigl(\log \delta^{-1}\bigr)^{d},
\label{ineq moc}
\end{equation}
where $\alpha_{d}=c_{1}\exp(3d^{2}\log d)$.
\end{lemma}

\begin{pf}
Since $f$ is continuous, for all $\theta\in\partial D_{\delta}$ we
have $%
f(\theta)=\delta$. By the polar integration formula (\ref{polar formula})
and inequality (\ref{uniform decay rate}),
\begin{eqnarray*}
\mu\bigl(\mathbb{R}^{d}\setminus D_{\delta}\bigr) &=&\int
_{\mathbb
{R}^{d}\setminus
D_{\delta}}f(x)\,dx
\\
&=&d\int_{1}^{\infty}\int_{\partial D_{\delta}}r^{d-1}f(r
\theta)\,d\mu _{D_{\delta}}(\theta)\,dr
\\
&\leq&d\int_{1}^{\infty}\int_{\partial D_{\delta}}r^{d-1}
\delta\exp \bigl(-c^{d}d^{-d/2}(r-1)\Vert \theta\Vert
_{2}\bigr)\,d\mu_{D_{\delta}}(\theta)\,dr.
\end{eqnarray*}
By (\ref{rigid 2}) and the fact that $\delta<2^{-8d}$, we have $%
1/9B_{2}^{d}\subset D_{\delta}$. By Corollary~\ref{vol growth},
\begin{eqnarray*}
\mu\bigl(\mathbb{R}^{d}\setminus D_{\delta}\bigr) &\leq&d\int
_{1}^{\infty
}\int_{\partial D_{\delta}}r^{d-1}
\delta\exp \bigl(-c^{d}d^{-d/2}(r-1)9^{-1}\bigr)\,d
\mu_{D_{\delta}}(\theta)\,dr
\\
&\leq&\delta\operatorname{vol}_{d}(D_{\delta})d\int
_{1}^{\infty
}r^{d-1}\exp \bigl(-c_{2}^{d}d^{-d/2}(r-1)
\bigr)\,dr
\\
&\leq&\beta_{d}\delta\bigl(\log\delta^{-1}
\bigr)^{d}d\exp\bigl(d^{2}\log d+c_{3}\bigr),
\end{eqnarray*}
where%
\[
\beta_{d}=\int_{1}^{\infty}r^{d-1}
\exp\bigl(-c_{2}^{d}d^{-d/2}(r-1)\bigr)\,dr.
\]
Set $\omega_{d}=c_{2}^{d}d^{-d/2}$ and $t=\omega_{d}r$. Recall the
definition of the gamma function $\Gamma(z)=\int_{0}^{\infty
}e^{-r}r^{z-1}\,dr$. By a change of variables and Stirling's formula,%
\begin{eqnarray*}
\beta_{d} &\leq&\exp(\omega_{d})\int_{0}^{\infty}r^{d-1}
\exp (-\omega _{d}r)\,dr
\\
&\leq&c_{4}\omega_{d}^{-d}\int
_{0}^{\infty}t^{d-1}e^{-t}\,dt
\\
&\leq&c_{5}\exp\bigl(d^{2}\log d\bigr)
\end{eqnarray*}
from which the result follows.
\end{pf}

Lemma~\ref{mass outside contour} is optimal in $\delta$ up to a factor
$%
\log\delta^{-1}$ as can be seen from the example $f(x)=\widetilde{c}%
_{d}\exp(-\Vert x\Vert _{2})$, in which case $\mu(\mathbb{R}^{d}\setminus
D_{\delta})\geq c_{d}\delta(\log\delta^{-1})^{d-1}$ for $\delta
<\delta
_{0}(d)$. To\vadjust{\goodbreak} see this, apply the polar integration formula as in the proof
of Lemma~\ref{mass outside contour} and use the equation%
\[
\int_{R}^{\infty}r^{d-1}e^{-r}\,dr=
\bigl(1+o_{d}(1)\bigr)R^{d-1}e^{-R}
\]
with $d$ fixed and $R\rightarrow\infty$.

\begin{lemma}
\label{log vs sqrt}There exists a universal constant $c>0$ such that
for all
$d\in\mathbb{N}$, if $t>d^{5d}$, then $\sqrt{t}\geq c(\log t)^{d}$.
\end{lemma}

Note that the inequality fails for $t=d^{2d}$.

\begin{pf*}{Proof of Lemma~\ref{log vs sqrt}}
First, consider any $d>16$. For any such $d$, $(2d)^{4d}<d^{5d}$. Set $T=2d$
and $x=\log t$. Since $(2d)^{4d}<d^{5d}<t$, it follows that $2T\log
T<x$. By
Lemma~\ref{little inequality}, $(\log x)/x<T^{-1}$, or equivalently%
\[
\frac{\log\log t}{\log t}<\frac{1}{2d},
\]
which is in turn equivalent to $\sqrt{t}>(\log t)^{d}$. By elementary
analysis, the number%
\[
c^{\prime}=\inf\bigl\{t^{1/2}(\log t)^{-d}\dvtx d
\leq16,t>d^{5d}\bigr\}
\]
is strictly positive. The result now follows for all $d\in\mathbb{N}$
with $%
c=\min\{c^{\prime},1\}$.
\end{pf*}

\begin{lemma}
\label{mass outside big contour}There exists a universal constant $%
\widetilde{c}>0$ with the following property. Let $d\in\mathbb{N}$ and
let $%
\mu$ be an isotropic log-concave probability measure with a continuous
density function $f$. For all $\delta<\widetilde{c}\exp(-8d^{2}\log
d)$,%
\[
\mu\bigl(\mathbb{R}^{d}\setminus2D_{\tau^{-1}9^{d}\delta}\bigr)<\delta,
\]
where $\tau=\tau_{d}=\operatorname{vol}_{d-1}(B_{2}^{d-1})%
\int_{1/2}^{1}(1-t^{2})^{(d-1)/2}\,dt$.
\end{lemma}

\begin{pf}
Consider the quantity $\alpha_{d}=c_{1}\exp(3d^{2}\log d)$ from Lemma
\ref%
{mass outside contour}. By concavity, $1-t^{2}\geq3(1-t)/2$ for all $%
1/2\leq t\leq3/4$. By a change of variables and Lemma~\ref{vol ball}, one
sees that $\tau>c_{2}^{d}d^{-d/2}$. Let $\kappa=\tau^{-1}9^{d}\delta$.
Consider any $y\in\partial(2D_{\kappa})$. Then $x=y/2\in\partial
D_{\kappa}$, and we have the convex combination $x=\frac{1}{2}0+\frac
{1}{2}y$. By log-concavity, $f(x)\geq f(0)^{1/2}f(y)^{1/2}$ and by inequality
(\ref%
{rigid 1}),%
\[
f(y)\leq\frac{f(x)^{2}}{f(0)}<2^{8d}\kappa^{2}
\]
and $y\notin D_{\varepsilon}$ with $\varepsilon=2^{8d}\kappa^{2}$. Since
this is true for all $y\in\partial(2D_{\kappa})$, $D_{\varepsilon
}\subset2D_{\kappa}$. For a sufficiently small choice of $\widetilde{c}$
(chosen independently of $d$), $\varepsilon<e^{-10d\log d-7}$. By
Lemma \ref%
{mass outside contour} and the inequality $e^{9d}\leq2^{8d}9^{2d}\leq
e^{10d}$,%
\begin{eqnarray*}
\mu\bigl(\mathbb{R}^{d}\setminus2D_{\kappa}\bigr) &\leq&\mu
\bigl(\mathbb{R}%
^{d}\setminus D_{\varepsilon}\bigr)\leq
\alpha_{d}\varepsilon\bigl(\log \varepsilon^{-1}
\bigr)^{d}
\\
&\leq&e^{10d}\alpha_{d}\tau^{-2}
\delta^{2}\bigl(2\log\delta^{-1}+\log \bigl(e^{-8d}
\tau^{2}\bigr)\bigr)^{d}
\\
&\leq&\delta\bigl(c^{-1}\delta^{1/2}\alpha_{d}2^{d}e^{10d}\tau^{-2}
\bigr)c\delta ^{1/2}\bigl(\log\delta^{-1}\bigr)^{d},
\end{eqnarray*}
where $c$ is the constant from Lemma~\ref{log vs sqrt}. By the bound imposed
on $\delta$, $c^{-1}\delta^{1/2}\alpha_{d}2^{d}e^{10d}\tau^{-2}<1$. The
result now
follows from Lemma~\ref{log vs sqrt}.
\end{pf}

Recall that $\mathcal{E}_{K}$ denotes the John ellipsoid of a convex
body $K$
and that $\mathcal{E}_{K}\subset K\subset d(\mathcal{E}_{K}-x_{0})+x_{0}$,
where $x_{0}$ is the center of $\mathcal{E}_{K}$.

\begin{lemma}
\label{little geometric}Let $K\subset\mathbb{R}^{d}$ be a convex body
with $%
0\in K$. Then $2K\subset3d(\mathcal{E}_{K}-x_{0})+x_{0}$.
\end{lemma}

\begin{pf}
By applying a suitable linear transformation, we may assume that
$\mathcal{E}%
_{K}=B_{2}^{d}+x_{0}$. Take any $x\in K$. Since $\max
\{\Vert x_{0}-x\Vert _{2},\Vert x_{0}\Vert _{2}\}\leq d$, it follows that $\Vert x\Vert _{2}\leq
\Vert x_{0}\Vert _{2}+\Vert x-x_{0}\Vert _{2}\leq2d$ and that $\Vert x_{0}-2x\Vert _{2}\leq
\Vert x_{0}-x\Vert _{2}+\Vert x-2x\Vert _{2}\leq3d$.
\end{pf}

\begin{lemma}
\label{ellipsoid disjoint}Let $\mathcal{E}$ be an ellipsoid with
centroid $%
\mathcal{O}$, and let $\mathfrak{H}$ be a half-space with $\mathrm
{vol}_{d}(%
\mathfrak{H}\cap\mathcal{E})\times\operatorname{vol}_{d}(B_{2}^{d})<\tau
_{d}%
\operatorname{vol}_{d}(\mathcal{E})$. Then $\mathfrak{H}$ and $\frac{1}{2}(%
\mathcal{E}-\mathcal{O})+\mathcal{O}$ are disjoint.
\end{lemma}

\begin{pf}
The truth of the lemma is invariant under affine transformations of~$%
\mathcal{E}$, and we may therefore assume that $\mathcal{E}=B_{2}^{d}$. The
result now follows from the definition of $\tau_{d}$ [see equation
(\ref%
{tau def})] and the fact that $\tau_{d}=\operatorname{vol}_{d}\{x\in\mathbb
{R}%
^{d}\dvtx \Vert x\Vert _{2}\leq1$, $x_{1}\geq1/2\}$.
\end{pf}

\begin{pf*}{Proof of Lemma \protect\ref{basic positioning}}
Consider $\tau=\tau_{d}$ defined by (\ref{tau def}). We may assume
that $%
\mu$ is in isotropic position, which implies that $D_{\delta
}^{\natural
}=D_{\tau^{-1}9^{d}\delta}$. Lemmas~\ref{mass outside big contour}
and~\ref{little geometric} together imply that $\mathfrak{H}\cap
\mathcal{E%
}_{\delta}^{\sharp}\neq\varnothing$. For each $x\in D_{\delta
}^{\natural
} $, $f(x)\geq\tau^{-1}9^{d}\delta$. Therefore $\delta\geq\mu(%
\mathfrak{H}\cap\mathcal{E}_{D_{\delta}^{\natural}})\geq\tau
^{-1}9^{d}\delta\operatorname{vol}_{d}(\mathfrak{H}\cap\mathcal
{E}_{D_{\delta
}^{\natural}})$ which implies that $\operatorname{vol}_{d}(\mathfrak{H}\cap
\mathcal{E}_{D_{\delta}^{\natural}})\leq\tau9^{-d}$. Since the density
function $f$ is continuous and $\tau^{-1}9^{d}\delta<2^{-8d}$,
inequality (%
\ref{rigid 2}) implies that $(9^{-1}+\kappa)B_{2}^{d}\subset D_{\delta
}^{\natural}$ for some $\kappa>0$. Since $\mathcal{E}_{D_{\delta
}^{\natural}}$ is the unique ellipsoid of maximal volume inside
$D_{\delta
}^{\natural}$, we have $\operatorname{vol}_{d}(\mathcal{E}_{D_{\delta
}^{\natural
}})>9^{-d}\operatorname{vol}_{d}(B_{2}^{d})$. From the definition of $\mathcal
{E}%
_{\delta}^{\flat}$ and Lemma~\ref{ellipsoid disjoint}, we see that $%
\mathfrak{H}\cap\mathcal{E}_{\delta}^{\flat}=\varnothing$. Finally, the
claim that $\mathcal{E}_{\delta}^{\flat}\subset F_{\delta}$ follows from
the definition of $F_{\delta}$ while the claim that $F_{\delta}\subset
\mathcal{E}_{\delta}^{\sharp}$ follows from the Hahn--Banach theorem
[any $%
x\notin\mathcal{E}_{\delta}^{\sharp}$ lies in an open half-space $%
\mathfrak{H}$ with $\mathfrak{H}\cap\mathcal{E}_{\delta}^{\sharp
}=\varnothing$ and therefore $\mu(\mathfrak{H})<\delta$ and $x\notin
F_{\delta}$].\vadjust{\goodbreak}~%
\end{pf*}

\begin{pf*}{Proof of Lemma \protect\ref{epsilon net positioning}}
Note that $1+\rho\leq(1-\rho)^{-1}\leq1+2\rho$ and $1-\rho\leq
(1+\rho
)^{-1}\leq1-\rho/2$, and the same inequalities hold for $\varepsilon$.
Since $rB_{2}^{d}\subset K\subset RB_{2}^{d}$, we have that%
\[
R^{-1}\Vert x\Vert _{2}\leq\Vert x\Vert _{K}\leq
r^{-1}\Vert x\Vert _{2}
\]
for all $x\in\mathbb{R}^{d}$. Combining this with (\ref{tight KL bound in
net}) gives%
\[
R^{-1}(1+\rho)^{-1}\leq\Vert \omega\Vert _{L}\leq
r^{-1}(1-\rho)^{-1}
\]
for all $\omega\in\mathcal{N}$. Consider any $\theta\in S^{d-1}$.
Let $%
\omega_{0}$ be the element of $\mathcal{N}$ that minimizes $\Vert \theta
-\omega_{0}\Vert _{2}$, and consider the series representation (\ref{series}).
By the triangle inequality,%
\[
\Vert \theta\Vert _{L}\leq r^{-1}(1-\rho)^{-1}(1-
\varepsilon)^{-1}.
\]
Hence $\Vert x\Vert _{L}\leq r^{-1}(1-\rho)^{-1}(1-\varepsilon
)^{-1}\Vert x\Vert _{2}$ for
all $x\in\mathbb{R}^{d}$.  Using the triangle inequality in a bit of a
different way,%
\begin{eqnarray*}
\Vert \theta\Vert _{L} &\geq&\Vert \omega_{0}\Vert
_{L}-\sum_{i=1}^{\infty
}\varepsilon
_{i}\Vert \omega_{i}\Vert _{L}
\\
&\geq&R^{-1}(1+\rho)^{-1}-r^{-1}\varepsilon(1-
\varepsilon )^{-1}(1-\rho )^{-1}
\\
&\geq&R^{-1}/2-4r^{-1}\varepsilon
\\
&=&R^{-1}\bigl(1-8Rr^{-1}\varepsilon\bigr)/2
\\
&\geq&(4R)^{-1},
\end{eqnarray*}
which holds since $8Rr^{-1}\varepsilon\leq1/2$.  Thus%
\begin{eqnarray*}
\Vert \theta\Vert _{L} &\leq&\Vert \omega_{0}\Vert
_{L}+\Vert \theta-\omega_{0}\Vert _{L}
\\
&\leq&(1-\rho)^{-1}\Vert \omega_{0}\Vert
_{K}+r^{-1}(1-\rho )^{-1}(1-\varepsilon
)^{-1}\varepsilon
\\
&\leq&(1-\rho)^{-1}\bigl(\Vert \theta\Vert _{K}+\Vert
\omega_{0}-\theta \Vert _{K}\bigr)+r^{-1}(1-
\rho)^{-1}(1-\varepsilon)^{-1}\varepsilon
\\
&\leq&(1-\rho)^{-1}\Vert \theta\Vert _{K}+r^{-1}(1-
\rho)^{-1}\varepsilon \bigl(1+(1-\varepsilon)^{-1}\bigr)
\\
&\leq&(1-\rho)^{-1}\Vert \theta\Vert _{K}+Rr^{-1}(1-
\rho)^{-1}\varepsilon \bigl(1+(1-\varepsilon)^{-1}\bigr)\Vert
\theta\Vert _{K}
\\
&\leq&(1+2\rho) \bigl(1+3Rr^{-1}\varepsilon\bigr)\Vert \theta\Vert
_{K}
\\
&\leq&\bigl(1+2\rho+6Rr^{-1}\varepsilon\bigr)\Vert \theta\Vert
_{K}.
\end{eqnarray*}
On the other hand,%
\begin{eqnarray*}
\Vert \theta\Vert _{K} &\leq&\Vert \omega_{0}\Vert
_{K}+\Vert \theta-\omega_{0}\Vert _{K}
\\
&\leq&(1+\rho)\Vert \omega_{0}\Vert _{L}+r^{-1}
\varepsilon
\\
&\leq&(1+\rho) \bigl(\Vert \theta\Vert _{L}+\Vert \omega_{0}-
\theta \Vert _{L}\bigr)+r^{-1}\varepsilon
\\
&\leq&(1+\rho)\Vert \theta\Vert _{L}+r^{-1}(1+\rho) (1-\rho
)^{-1}(1-\varepsilon )^{-1}\varepsilon+r^{-1}
\varepsilon
\\
&\leq&(1+\rho)\Vert \theta\Vert _{L}+7r^{-1}\varepsilon
\cdot4R\Vert \theta \Vert _{L}
\\
&\leq&\bigl(1+\rho+28Rr^{-1}\varepsilon\bigr)\Vert \theta\Vert
_{L}.
\end{eqnarray*}
The result follows by positive homogeneity.
\end{pf*}

\section{\texorpdfstring{Proof of Theorem \protect\ref{main 3}}{Proof of Theorem 3}}\label{other approximants}
Fix $\mu$ and $d$ as in the statement of Theorem~\ref{main 3}. Let~$f$ be
the density of $\mu$, and let $g=-\log f$. All variables in this section
depend on both $d$ and $\mu$.

\begin{lemma}
There exist $c,\varepsilon_{0}>0$ such that for all $\varepsilon\in
(0,\varepsilon_{0})$,%
%
\begin{equation}
\mu\bigl(\mathbb{R}^{d}\setminus D_{\varepsilon}\bigr)<c\varepsilon
\bigl( \log \varepsilon^{-1} \bigr) ^{d} \label{fundamental
function}.
\end{equation}
\end{lemma}

\begin{pf}
Since $\mu$ has a log-concave density, it necessarily has a nonsingular
covariance matrix, and there exists an affine map $T$ such that $\mu
^{\prime}=T\mu$ is isotropic. The density of $\mu^{\prime}$ is%
\[
\widetilde{f}(x)=\det\bigl(T^{-1}\bigr)f\bigl(T^{-1}x\bigr)
\]
and $D_{\varepsilon}=T^{-1}\widetilde{D}_{\widetilde{\varepsilon}}$, where
$\widetilde{\varepsilon}=\varepsilon\det T^{-1}$ and $\widetilde{D}_{%
\widetilde{\varepsilon}}=\{x\dvtx \widetilde{f}(x)\geq\widetilde
{\varepsilon}%
\} $. Since $\mu^{\prime}$ is isotropic, we may use Lemma~\ref{mass
outside contour}, which gives
\begin{eqnarray*}
\mu\bigl(\mathbb{R}^{d}\setminus D_{\varepsilon}\bigr) &=&
\mu^{\prime
}\bigl(\mathbb{R}%
^{d}\setminus
\widetilde{D}_{\widetilde{\varepsilon}}\bigr)
\\
&\leq&c^{\prime}\widetilde{\varepsilon}\bigl(\log\widetilde{\varepsilon}
^{-1}\bigr)^{d}
\\
&<&c\varepsilon \bigl( \log\varepsilon^{-1} \bigr) ^{d}.
\end{eqnarray*}
\upqed\end{pf}

\begin{lemma}
For any $x\in\mathbb{R}^{d}$ there exist $c^{\prime},\delta_{0}>0$
and a
function $p\dvtx (0,\delta_{0})\rightarrow(0,\infty)$ such that for all $%
\delta\in(0,\delta_{0})$,%
%
\begin{equation}
p(\delta)\leq c^{\prime}\frac{\log\log\delta^{-1}}{\log\delta^{-1}} \label{pfunc}
\end{equation}
and%
%
\begin{equation}
F_{\delta}\subset(1+p) (D_{\delta}-x)+x \label{outer}.
\end{equation}
\end{lemma}

\begin{pf}
Let $c>0$ be the constant in (\ref{fundamental function}). A brief analysis
of the function $t\mapsto ct ( \log t^{-1} ) ^{d}$ shows that there
exists $\delta_{0}>0$ and a function $\varepsilon=\varepsilon(\delta)$
defined implicitly for all $\delta\in(0,\delta_{0})$ by the equation
$%
\delta=c\varepsilon ( \log\varepsilon^{-1} ) ^{d}$. We can take
$\delta_{0}$ small enough to ensure that $\varepsilon<\delta$ and
that $%
\log\delta^{-1}<\log\varepsilon^{-1}<2\log\delta^{-1}$. If we
define%
\[
p(\delta)=3\frac{\log\varepsilon^{-1}-\log\delta^{-1}}{\log\delta
^{-1}%
},
\]
then $\delta^{1+p/2}<\varepsilon$, and (\ref{pfunc}) holds. Since $%
D_{\varepsilon}$ is both compact and convex, for any point $y\notin
D_{\varepsilon}$ there exists (by the Hahn--Banach theorem),\vadjust{\goodbreak} a closed
half-space $\mathfrak{H}$ with $y\in\mathfrak{H}$ and $\mathfrak{H}\cap
D_{\varepsilon}=\varnothing$. Since $\mathfrak{H}\subset\mathbb{R}%
^{d}\setminus D_{\varepsilon}$, (\ref{fundamental function}) implies
that $%
\mu(\mathfrak{H})<\delta$ and by definition of $F_{\delta}$, $y\notin
F_{\delta}$. This goes to show that $F_{\delta}\subset D_{\varepsilon
}$. Let $x\in\mathbb{R}^{d}$. For any $\theta\in S^{d-1}$, consider the
function $f_{\theta}(t)=f(x+t\theta)=e^{-g_{\theta}(t)}$, $t\in
\mathbb{R}
$. This notation differs slightly from that in the proof of Theorem 1. By
continuity and log-concavity, if $\varepsilon$ is small enough, then
for all
$\theta\in S^{d-1}$ there is a unique $v>0$ such that \mbox{$f_{\theta
}(v)=\varepsilon$}; we denote this number by $f_{\theta
}^{-1}(\varepsilon)$. We may assume that $\delta_{0}<\min\{1,f(x)^{2}\}$. Note that
$1<\delta
/\varepsilon<\delta^{-p/2}$ and $\log\delta^{-1}+\log f(x)\geq
1/2\log
\delta^{-1}$. By convexity of $g_{\theta}$, for any $0<s<v$ we have $%
s^{-1}(g_{\theta}(s)-g_{\theta}(0))\leq(v-s)^{-1}(g_{\theta
}(v)-g_{\theta}(s))$. Taking $v=f_{\theta}^{-1}(\varepsilon)$ and $%
s=f_{\theta}^{-1}(\delta)$, this becomes%
%
\begin{eqnarray}
\label{yet another inequality} \frac{f_{\theta}^{-1}(\varepsilon)-f_{\theta}^{-1}(\delta
)}{f_{\theta
}^{-1}(\delta)} &\leq&\frac{\log\varepsilon^{-1}-\log\delta
^{-1}}{\log
\delta^{-1}+\log f(x)}.
\nonumber
\\[-8pt]
\\[-8pt]
\nonumber
&<&p.
\end{eqnarray}
Inequality (\ref{yet another inequality}) reduces to $f_{\theta
}^{-1}(\varepsilon)\leq(1+p)f_{\theta}^{-1}(\delta)$. Since this holds
for any $\theta\in S^{d-1}$, $D_{\varepsilon}\subset(1+p)(D_{\delta
}-x)+x $, and (\ref{outer}) follows.
\end{pf}

\begin{lemma}
There exists $\delta_{0}>0$ such that for all $\delta\in(0,\delta_{0})$
we have the relation%
%
\begin{equation}
\bigl(1+8\lambda^{1/d}\bigr)^{-1}\bigl(D_{\delta}-x^{\prime}
\bigr)+x^{\prime}\subset F_{\delta}, \label{inner}
\end{equation}
where $\lambda=\operatorname{vol}_{d}(D_{\delta})^{-1}$, and $x^{\prime}$
is the
centroid of $D_{\delta}$.
\end{lemma}

\begin{pf}
Let $\delta_{0}$ be such that $\operatorname{vol}_{d}(D_{\delta
_{0}})>8^{d}$. We
use the notation $(D_{\delta})_{\lambda}$ for the convex floating body
with parameter $\lambda>0$ corresponding to the uniform probability measure
on $D_{\delta}$. If $\mathfrak{H}$ is any half-space with $\mu
(\mathfrak{H}%
)<\delta$, then $\operatorname{vol}_{d}(\mathfrak{H}\cap D_{\delta})<1$.
Hence $%
(D_{\delta})_{\lambda}\subset F_{\delta}$, where $\lambda=\mathrm
{vol}%
_{d}(D_{\delta})^{-1}$. The result now follows from inequality (\ref{Fr
bound}).
\end{pf}

\begin{lemma}
Let $K,L\subset\mathbb{R}^{d}$ be convex bodies such that there exist
$x,x^{\prime}\in \operatorname{int}(K\cap L)$ and $0<r<(8d)^{-1}$ for which
%
\begin{equation}
(1+r)^{-1}(K-x)+x\subset L\subset(1+r) \bigl(K-x^{\prime}
\bigr)+x^{\prime}. \label{almost geometric}
\end{equation}
Then%
%
\begin{equation}
d_{\mathfrak{L}}(K,L)\leq1+8\,dr. \label{skew geometric}
\end{equation}
\end{lemma}

\begin{pf}
Since the statement of the lemma is invariant under affine transformations
that act simultaneously on $K$ and $L$, we may assume without loss of
generality that the John ellipsoid of $K$ is $B_{2}^{d}$. Hence $%
B_{2}^{d}\subset K\subset  dB_{2}^{d}$ and $\Vert x\Vert _{2}$,$\Vert x^{\prime
}\Vert _{2}\leq d$.  Note also that $L\subset3 dB_{2}^{d}$.  Using these facts
and manipulating (\ref{almost geometric}) in the\vadjust{\goodbreak} obvious way, we see that
both of the following relations hold:%
\begin{eqnarray*}
L &\subset&K+2 drB_{2}^{d},
\\
K &\subset&L+4 drB_{2}^{d}.
\end{eqnarray*}
By definition of the Hausdorff distance, $d_{\mathcal{H}}(K,L)\leq4 dr<1/2$.
Since $B_{2}^{d}\subset K$, inequality (\ref{LH}) implies that
$d_{\mathfrak{%
L}}(K,L)\leq1+8 dr$.
\end{pf}

\begin{pf*}{Proof of equation (\ref{FD1})}
Since $\lim_{\delta\rightarrow0}p(\delta)=\lim_{\delta\rightarrow
0}\lambda(\delta)=0$, equation~(\ref{FD1}) now follows from (\ref{outer}),
(\ref{inner}) and (\ref{skew geometric}).
\end{pf*}

\begin{remark}
\label{Remark only}There is no lower bound on the growth rate of
$\mathrm{vol%
}_{d}(D_{\delta})$; indeed the function could grow arbitrarily
slowly.  However in the case of the Schechtman--Zinn distributions, $\operatorname{vol}
_{d}(D_{\delta})=(\log(c_{p}^{d}/\delta))^{d/p}\operatorname{vol}%
_{d}(B_{p}^{d}), $ and we leave it to the reader to combine this with
(\ref%
{outer}), (\ref{pfunc}) and (\ref{inner}) to obtain a quantitative upper
bound on $d_{\mathfrak{L}}(F_{\delta},D_{\delta})$.
\end{remark}

\begin{pf*}{Proof of equation (\protect\ref{FR1})}
Let $\varepsilon>0$ be given. Using the notation from the proof of Theorem
1, for any $\theta\in S^{d-1}$ we define%
\[
f_{\theta}(t)=-\frac{d}{dt}\mu(\mathfrak{H}_{\theta,t}).
\]
This function is the density of a log-concave probability measure on $%
\mathbb{R}$ with cumulative distribution function $J_{\theta}(t)=1-\mu
(%
\mathfrak{H}_{\theta,t})$. Using Fubini's theorem we have%
\[
f_{\theta}(t)=Rf(\mathcal{H}_{\theta,t}),
\]
where $Rf$ is the Radon transform of $f$ as defined by (\ref{Radon def}).
For all $t\in\mathbb{R}$, the function $\theta\mapsto f_{\theta}(t)$ is
continuous and nonvanishing on $S^{d-1}$. This follows using the properties
imposed on $f$, together with Lebesgue's dominated convergence theorem
and (%
\ref{domination of cone}). Define $\alpha=\inf\{f_{\theta}(0)\dvtx \theta
\in
S^{d-1}\}$, and note that $\alpha>0$. By (\ref{domination of cone}) there
exists $t_{0}>0$ such that if $\beta=\sup\{f_{\theta}(t_{0})\dvtx \theta
\in
S^{d-1}\}$, then $\beta<\alpha$. Define $g_{\theta}(t)=-\log
f_{\theta
}(t)$, and let $\lambda=t_{0}^{-1}(\log\alpha-\log\beta)$ and
$\Delta
=\max\{1,\lambda^{-1}\log\lambda^{-1}\}$. By definition of $\alpha
$, $%
\beta$ and $\lambda$, for all $\theta\in S^{d-1}$ we have $%
t_{0}^{-1}(g_{\theta}(t_{0})-g_{\theta}(0))\geq\lambda$. By
convexity of
$g_{\theta}$, if $u>v\geq t_{0}$, then $g_{\theta}(u)\geq g_{\theta
}(v)+\lambda(u-v)$, which translates into $f_{\theta}(u)\leq f_{\theta
}(v)e^{-\lambda(u-v)}$. Let $\delta_{0}<\inf\{f_{\theta
}(t_{0}+1)\dvtx \theta
\in S^{d-1}\}$ be such that $\Delta\varepsilon^{-1}B_{2}^{d}\subset
F_{\delta_{0}}$. Consider any $\delta<\delta_{0}$, and momentarily
fix $%
\theta\in S^{d-1}$. By definition of $\alpha$ and $\beta$,%
\[
0<f_{\theta}(t_{0})\leq\beta<\alpha\leq f_{\theta}(0).
\]
Since $f_{\theta}$ is log-concave, it must be strictly decreasing on $%
[t_{0},\infty)$. Let $s=J_{\theta}^{-1}(1-\delta)$ and denote by $%
t=f_{\theta}^{-1}(\delta)$ the unique positive number such that
$f_{\theta
}(t)=\delta$. Consider the hyperplane $\mathcal{H}_{\theta,t}$ and the
half-space $\mathfrak{H}_{\theta,s}$. Note that%
\[
\mu(\mathfrak{H}_{\theta,s})=Rf(\mathcal{H}_{\theta,t})=\delta.
\]
By definition of $\delta_{0}$ and the equation $f_{\theta
}(-u)=f_{-\theta
}(u)$, we have $\min\{f_{\theta}(t_{0}+1),f_{\theta}(-t_{0}-1)\}
>\delta
_{0}$. By log-concavity we have $f_{\theta}(u)\geq\delta_{0}$ for
all $%
-t_{0}-1<u<t_{0}+1$, hence $t>t_{0}+1$. By the fundamental theorem of
calculus and the fact that $f_{\theta}(u)\geq\delta$, for all $u\in
\lbrack t-1,t]$, we have
\begin{eqnarray*}
\mu(\mathfrak{H}_{\theta,t-1}) &>&\mu\bigl\{x\in\mathbb{R}^{d}\dvtx
t-1\leq \langle\theta,x \rangle\leq t\bigr\}
\\
&=&\int_{t-1}^{t}f_{\theta}(u)\,du
\\
&\geq&\delta.
\end{eqnarray*}
Hence $\mathfrak{H}_{\theta,s}\subset\mathfrak{H}_{\theta,t-1}$ and $
s>t-1>t_{0}$. Thus, if $s\leq t$, then $|s-t|\leq1$. If $s>t$, then%
\begin{eqnarray*}
\delta&=&\int_{s}^{\infty}f_{\theta}(u)\,du
\\
&\leq&f_{\theta}(s)\int_{s}^{\infty}e^{-\lambda(u-s)}\,du
\\
&\leq&\delta e^{-\lambda(s-t)}\lambda^{-1}
\end{eqnarray*}
from which it follows that $s-t\leq\lambda^{-1}\log\lambda^{-1}$. Either
way, $|s-t|\leq\max\{1,\break \lambda^{-1}\log\lambda^{-1}\}=\Delta$.
Since $%
\Delta\varepsilon^{-1}B_{2}^{d}\subset F_{\delta_{0}}$, it follows
that $%
s\geq\Delta\varepsilon^{-1}$ and $(1-\varepsilon)s\leq t\leq
(1+\varepsilon)s$. Since this holds for all $\theta\in S^{d-1}$, and
recalling the definitions of $F_{\delta}$ and $R_{\delta}$, we have%
\[
(1-\varepsilon)F_{\delta}\leq R_{\delta}\leq(1+\varepsilon
)F_{\delta}.
\]
\upqed\end{pf*}

\section{Optimality}\label{optimality section}

Let $\Phi$ denote the cumulative standard normal distribution on~$\mathbb{R}
$,%
\[
\Phi(t)=(2\pi)^{-1/2}\int_{-\infty}^{t}e^{-({1}/{2})s^{2}}\,ds.
\]
By (\ref{sandwitch}) there exists $c>0$ such that for all $n\geq3$,%
%
\begin{equation}
\Phi^{-1}(1-1/n)\geq c(\log n)^{1/2}. \label{pr}
\end{equation}

\begin{lemma}
For all~$q>0$ and all $d\in\mathbb{N}$, there exists $c,\widetilde{c}>0$
such that for all $n\geq d+1$, if $(x_{i})_{1}^{n}$ is an i.i.d. sample from
the standard normal distribution on $\mathbb{R}^{d}$ and $%
P_{n}=\operatorname{conv}\{x_{i}\}_{1}^{n}$, then with probability at least
$1-\widetilde{c}%
(\log n)^{-q(d-1)/2}$ both of the following events occur:%
%
\begin{eqnarray}
d_{\mathcal{H}}(P_{n},F_{1/n}) &\geq&c(\log
n)^{-({1}/{2})-q},\label{lower bound}
\\
d_{\mathfrak{L}}(P_{n},F_{1/n}) &\geq&1+c(\log
n)^{-1-q}. \label{lower bound g}
\end{eqnarray}
\end{lemma}

\begin{pf}The probability bound is trivial when $d = 1$ and we may assume that $d \geq 2$.
It follows from a result of Schneider~\cite{Sc}\vadjust{\goodbreak} (see also~\cite{Gr},
page 326)
that for any polytope $K_{m}\subset\mathbb{R}^{d}$ with at most $m$
vertices,%
%
\begin{equation}
d_{\mathcal{H}}\bigl(K_{m},B_{2}^{d}
\bigr)>c_{d} \biggl( \frac{1}{m} \biggr) ^{{2}/{%
(d-1)}}.
\label{approx H}
\end{equation}
Since $F_{1/n}=\Phi^{-1}(1-1/n)B_{2}^{d}$, inequality (\ref{pr}) implies
that%
\[
d_{\mathcal{H}}(K_{m},F_{1/n})>c_{d}(\log
n)^{1/2} \biggl( \frac
{1}{m} \biggr) ^{{2}/{(d-1)}}.
\]
By a result of Baryshnikov and Vitale~\cite{Bary} (see also~\cite{Ball2}), the number of vertices
of $P_{n}$, denoted by $f_{0}(P_{n})$, obeys the inequality $%
\mathbb{E}f_{0}(P_{n})<\widetilde{c}_{d}(\log n)^{(d-1)/2}$. By Chebyshev's
inequality we have%
\[
\mathbb{P}\bigl\{f_{0}(P_{n})>(\log n)^{{(d-1)(q+1)}/{2}}
\bigr\}\leq\frac
{\mathbb{%
E}f_{0}(P_{n})}{(\log n)^{{(d-1)(q+1)}/{2}}}<\widetilde{c}_{d}(\log n)^{-%
{q(d-1)}/{2}},
\]
and if the complement of this event occurs, then so does (\ref{lower bound}).
By (\ref{approx H}) and (\ref{HL}), we get%
\[
d_{\mathfrak{L}}\bigl(K_{m},B_{2}^{d}
\bigr)>1+c_{d} \biggl( \frac{1}{m} \biggr) ^{{2%
}/{(d-1)}}.
\]
Since $d_{\mathfrak{L}}$ is preserved by invertible affine transformations
[as per (\ref{af})], the same inequality holds for all Euclidean balls. This
gives (\ref{lower bound g}).
\end{pf}

We can choose $q$ to be arbitrarily small, in which case (\ref{lower bound})
and (\ref{lower bound g}) complement Theorems~\ref{main 1} and~\ref%
{main 2}.

\section{\texorpdfstring{Proof of Theorem \protect\ref{main 4}}{Proof of Theorem 4}}

Let $\mathcal{K}_{d}$ denote the collection of all convex bodies in~$\mathbb{%
R}^{d}$ (compact convex sets with nonempty interior), and let $\mathcal
{K}%
_{d}^{\sharp}=\mathcal{K}_{d}\cup\{\{0\}\}$. If $\Omega$ is a convex
subset of a real vector space, then we define a function $\kappa\dvtx \Omega
\rightarrow\mathcal{K}_{d}^{\sharp}$ to be concave if for all $x,y\in
\Omega$ and all $\lambda\in(0,1)$, we have%
\[
\lambda\kappa(x)+(1-\lambda)\kappa(y)\subset\kappa\bigl(\lambda x+(1-\lambda)y
\bigr).
\]
If $\Omega$ has an ordering, then we define $\kappa$ to be nondecreasing
if for all $x,y\in\Omega$ with $x\leq y$, we have $\kappa(x)\subset
\kappa
(y)$.

\begin{lemma}
If $\kappa\dvtx [0,\infty)\rightarrow\mathcal{K}_{d}^{\sharp}$ is concave,
nondecreasing and\break $\bigcup_{t\in\lbrack0,\infty)}\kappa(t)=\mathbb
{R}^{d}$, then the function $g\dvtx \mathbb{R}^{d}\rightarrow\lbrack0,\infty)$ defined
by%
%
\begin{equation}
g(x)=\inf\bigl\{t\geq0\dvtx x\in\kappa(t)\bigr\} \label{g}
\end{equation}
is convex. Furthermore, $\kappa$ is continuous on $(0, \infty)$ with respect to the
Hausdorff distance, and for all $t>0$,%
%
\begin{equation}
\kappa(t)=\bigl\{x\in\mathbb{R}^{d}\dvtx g(x)\leq t\bigr\}. \label{k
rep}
\end{equation}
\end{lemma}

\begin{pf}
By hypothesis, $\kappa(0)\neq\varnothing$. If $0\notin\kappa(0)$,
then we
define $\kappa^{\sharp}(t)=\kappa(t)-x_{0}$, where $x_{0}\in\kappa(0)$.
The function $\kappa^{\sharp}$ enjoys all of the properties that
$\kappa$
does, and the function%
\[
g^{\sharp}(x)=\inf\bigl\{t\geq0\dvtx x\in\kappa^{\sharp}(t)\bigr\}
\]
is related to $g$ by the equation $g^{\sharp}(x)=g(x+x_{0})$. Note
that $0\in\kappa^{\sharp}(0)$. If the lemma holds for the functions $\kappa
^{\sharp}$ and $g^{\sharp}$, it will necessarily hold for $\kappa$
and $g$. We may therefore, without loss of generality, assume that $0\in
\kappa(0)$. For any $0<\varepsilon<t$, we have the convex combination%
\[
t=\frac{\varepsilon}{t+\varepsilon}0+\frac{t}{t+\varepsilon}%
(t+\varepsilon).
\]
Exploiting the concavity of $\kappa$, this leads to%
\[
\kappa(t+\varepsilon)\subset\frac{t+\varepsilon}{t}\kappa(t).
\]
Similarly,%
\[
\frac{t-\varepsilon}{t}\kappa(t)\subset\kappa(t-\varepsilon).
\]
Hence $\kappa$ is continuous with respect to the Hausdorff distance. By
definition of $g$, $\kappa(t)\subset\{x\in\mathbb{R}^{d}\dvtx g(x)\leq t\}
$. Since $\kappa(t)$ is a closed set, if $x\notin\kappa(t)$, then
$d(x,\kappa
(t))>0$ and by continuity of $\kappa$, $g(x)>t$. This implies (\ref{k
rep}%
). Consider any $x,y\in\mathbb{R}^{d}$ and $\lambda\in(0,1)$.
Let $t = g(x)$ and $s = g(y)$. For all $t' > t$ and $s' > s$, $x\in \kappa(t')$ and $y \in \kappa(s')$.
Therefore
\begin{eqnarray*}
\lambda x + (1 -\lambda)y &\in& \lambda \kappa\bigl(t'\bigr) + (1
-\lambda)\kappa\bigl(s'\bigr)\\
&\subset& \kappa\bigl(\lambda t' + (1 +\lambda)s'\bigr)
\end{eqnarray*}
This implies that $g(\lambda x+(1-\lambda )y)\leq \lambda t' +(1-\lambda)s'$.
Since this holds for all such~$t'$ and $s'$, it follows that g is convex.
\end{pf}

Note that the function $g$ is a generalization of the Minkowski functional
of a convex body $K$, in which case $\kappa(t)=tK$. The converse of the
preceding lemma also holds; if we are given a convex function $g$ and define
$\kappa$ via (\ref{k rep}), then $\kappa$ is concave.

If $(K_{n})_{n=1}^{\infty}$ is a sequence of convex bodies such that the
partial Minkowski sums $S_{N}=\sum_{n=1}^{N}K_{n}$ converge in the Hausdorff
distance to a convex body $S$, then we write $S=\sum_{n=1}^{\infty}K_{n}$
and refer to this as a Minkowski series. Note that $S$ can also be defined
by the equation%
\[
\Vert x\Vert _{S^{\circ}}=\sum_{n=1}^{\infty}
\Vert x\Vert _{K_{n}^{\circ}}.
\]
Basic properties of a Minkowski series are easy to prove, and we leave such
an investigation to the reader. The following lemma is also fairly
straightforward.

\begin{lemma}
For each $n\in\mathbb{N}$, let $\alpha_{n}\dvtx [0,\infty)\rightarrow
\lbrack
0,\infty)$ be a concave function, and let $K_{n}$ be a convex body
with $%
0\in K_{n}$. Provided that%
\[
\sum_{n=1}^{\infty}\alpha_{n}(t)\operatorname{diam}(K_{n})<
\infty
\]
for all $t\geq0$, then the function $\kappa\dvtx [0,\infty)\rightarrow
\mathcal{K}_{d}^{\sharp}$ defined by%
\[
\kappa(t)=\sum_{n=1}^{\infty}
\alpha_{n}(t)K_{n}
\]
is concave.
\end{lemma}

The space $\mathcal{K}_{d}$ is separable with respect to $d_{\mathrm{BM}}$, and we
shall use a sequence $(K_{n})_{n=1}^{\infty}$ that is dense in
$\mathcal{K}%
_{d}$. Since $d_{\mathrm{BM}}$ is blind to affine transformations, we can assume that
the John ellipsoid of each $K_{n}$ is $B_{2}^{d}$. As coefficients, we shall
use the functions%
\[
\alpha_{n}(t)=\cases{
2^{-n^{2}}t, & \quad $0\leq t\leq2^{2n^{2}},$
\vspace*{2pt}\cr
2^{n^{2}}, & \quad $2^{2n^{2}}<t<\infty.$}
\]
Note that for large values of $n$, the dominant coefficient at the
value $%
t=2^{2n^{2}}$ is $\alpha_{n}$. In fact $\sum_{j\neq n}\alpha
_{j}(2^{2n^{2}})$ is much smaller than $\alpha_{n}(2^{2n^{2}})$,%
\begin{eqnarray*}
\sum_{j\neq n}\alpha_{j}
\bigl(2^{2n^{2}}\bigr) &=&\sum_{j=1}^{n-1}2^{j^{2}}+2^{2n^{2}}
\sum_{j=n+1}^{\infty}2^{-j^{2}}
\\
&\leq&\sum_{j=1}^{n-1}2^{nj}+2^{2n^{2}}
\sum_{j=n+1}^{\infty}2^{-nj}
\\
&\leq&2^{n^{2}-n+2}
\\
&=&2^{-n+2}\alpha_{n}\bigl(2^{2n^{2}}\bigr).
\end{eqnarray*}
Hence,%
\[
d_{\mathrm{BM}}\bigl(\kappa\bigl(2^{2n^{2}}\bigr),K_{n}\bigr)
\leq1+2^{-n+2}d.
\]
Thus the sequence $(\kappa(n))_{n=1}^{\infty}$ is dense in $\mathcal
{K}%
_{d} $. Since each coefficient $\alpha_{n}$ is nondecreasing and
concave, $%
\kappa$ is concave and the function $g$ as defined by (\ref{g}) is convex.
Clearly, $\lim_{x\rightarrow\infty}g(x)=\infty$.  For some $c>0$, the
function%
\[
f(x)=2^{-g(cx)}
\]
is the density of a log-concave probability measure $\mu$ on $\mathbb
{R}%
^{d} $. For each $n\in\mathbb{N}$, $D_{2^{-n}}=\{x\in\mathbb{R}%
^{d}\dvtx f(x)\geq2^{-n}\}=\{x\in\mathbb{R}^{d}\dvtx g(cx)\leq n\}=c^{-1}\kappa
(n)$%
. Hence the sequence $(D_{1/n})_{n=2}^{\infty}$ is dense in $\mathcal
{K}%
_{d} $. By (\ref{FD1}), the sequence $(F_{1/n})_{n=3}^{\infty}$ is also
dense in~$\mathcal{K}_{d}$.\vadjust{\goodbreak}

We now use Theorem~\ref{main 1} with $q=1$. Let $\widetilde{\mathcal{K}}_{d}$
denote a countably dense subset of $\mathcal{K}_{d}$, and let $K\in
\widetilde{\mathcal{K}}_{d}$. For any $\varepsilon>0$, there exists an
increasing sequence of natural numbers $(k_{n})_{1}^{\infty}$ such
that $%
\lim_{n\rightarrow\infty}d_{\mathrm{BM}}(F_{1/k_{n}},K)=1$ and $\sum_{n=1}^{\infty
}3^{d+3}(\log k_{n})^{-1}<\varepsilon$. By (\ref{main bound}),%
\[
\lim_{n\rightarrow\infty}d_{\mathrm{BM}}(P_{k_{n}},K)=1
\]
with probability at least $1-\varepsilon$. Since this holds for all $%
\varepsilon>0$, $K\in \operatorname{cl}_{\mathrm{BM}}\{P_{n}\dvtx\break n\in\mathbb{N}$, $n\geq d+1\}$ almost
surely, where $\operatorname{cl}_{\mathrm{BM}}$ denotes closure in $\mathcal{K}_{d}$ with
respect to
$d_{\mathrm{BM}}$. Since this holds for all $K\in\widetilde{\mathcal{K}}_{d}$
and $%
\widetilde{\mathcal{K}}_{d}$ is countable, $\widetilde{\mathcal{K}}%
_{d}\subset \operatorname{cl}_{\mathrm{BM}}\{P_{n}\dvtx n\in\mathbb{N}$, $n\geq d+1\}$ almost surely.
The result now follows since $\widetilde{\mathcal{K}}_{d}$ is dense in $
\mathcal{K}_{d}$.

\section*{Acknowledgments}
Many thanks to Imre B\'{a}r\'{a}ny, John Fresen, Jill Fresen, Nigel
Kalton, Alexander Koldobsky, Mathieu Meyer, Mark Rudelson, Michael Taksar
and Artem Zvavitch for their comments, advice and support.  In particular,
I am grateful to Nigel Kalton for his friendship and all that he taught me.
The anonymous referees made valuable suggestions that led to major
improvements in the paper, and for this I am also grateful. This paper
was written
while the author was a graduate student at the University of Missouri.

%


\printaddresses

\end{document}